\documentclass[12pt]{article}
\usepackage{amssymb,amsmath,amscd,graphicx,latexsym,amsthm,hyperref,float,xypic,simplewick}
\usepackage[utf8]{inputenc}
\usepackage{titlesec}
\usepackage{suffix}
\titleformat*{\section}{\normalsize\bfseries}
\titleformat*{\subsection}{\normalsize\bfseries}
\titleformat*{\subsubsection}{\small\bfseries}

\numberwithin{equation}{section}
\RequirePackage[dvipsnames,usenames]{xcolor}

\usepackage{hyperref}
\hypersetup{
    colorlinks,
    citecolor=BlueViolet,
    filecolor=BlueViolet,
    linkcolor=BlueViolet,
    urlcolor=blue
}

\usepackage{mathtools}

\usepackage{graphicx}
\newcommand*{\matminus}{%
  \leavevmode
  \hphantom{0}%
  \llap{%
    \settowidth{\dimen0 }{$0$}%
    \resizebox{1.1\dimen0 }{\height}{$-$}%
  }%
}

\newcounter{count}
\theoremstyle{plain}
\newtheorem{definition}[count]{Definition}
\newtheorem{lemma}[count]{Lemma}
\newtheorem{corollary}[count]{Corollary}
\newtheorem{proposition}[count]{Proposition}

\newtheorem{theorem}[count]{Theorem}

\newtheorem{observation}[count]{Observation}

\theoremstyle{remark}
\newtheorem*{remark}{Remark}
\newtheorem{remarkN}[count]{Remark}
\newtheorem{example}[count]{Example}

\renewcommand{\P}{\mathbb{P}}
\renewcommand{\O}{\mathcal{O}}
\renewcommand{\L}{\Lambda}
\renewcommand{\ker}{\operatorname{ker}}
\newcommand{\im}{\operatorname{image}}
\newcommand{\GL}{GL}
\newcommand{\SL}{SL}

\newcommand{\stab}{\operatorname{stab}}

\begin{document}

\title{Multi-focal tensors as invariant differential forms}
\author{\small James Mathews \\
% \small Stony Brook University \\
% \small Math Tower 2-130, Stony Brook, NY \\
\small \href{mailto:jmath@math.stonybrook.edu}{jmath@math.stonybrook.edu} }
\date{}
\maketitle

\begin{abstract}
For each relative $\GL(V)$-invariant tensor $I\in \L^{p_1+1}V^{\vee}\otimes .. \otimes \L^{p_n+1}V^{\vee}$ we construct a $\GL(V)$-invariant weighted differential form $\eta$ on $(\P V)^{n}$.
Then $\eta$ is expressed explicitly with respect to $n$-tuples of frames for tangent spaces at points of $\P V$ to obtain elements of a different tensor space.

For certain invariants $I$, the resulting elements are shown to be the multi-focal tensors appearing in the machine vision literature (\cite{dmz}, \cite{lh}, \cite{luong}, \cite{faug93}, \cite{lf2}, \cite{fl}, \cite{hz}).
This generalizes the 3 multi-focal varieties known in dimension $\dim V = 4$ to an infinite collection of special tensor subvarieties.
We use this framework to exhibit a new system of degree 4 polynomial equations, reminiscent of the braid relation, satisfied by the Euclidean trifocal variety.
\end{abstract}

\newpage
\tableofcontents
\small
\newpage

\subsection{Notation}

We will use the following notation:

\begin{itemize}
\item{$k$, a field of characteristic zero (this condition on $k$ is required in Lemma \ref{char0})}
\item{$V$, a finite-dimensional vector space defined over $k$}
\item{$V^{\vee}$, the $k$-linear dual of $V$}
\item{$\P V$, the projective space of $V$, the quotient space of $V\backslash\{0\}$ by the action of the multiplicative group $k^{\times}$, with the Zariski topology}
\item{$\O_X$, the sheaf of regular functions on a variety $X$ (abbreviated $\O$)}
\item{$\O_{\P V}$, the sheaf of regular functions on $\P V$}
\item{$\mathcal{O}(1)$, the sheaf of algebraic sections of the dual of the tautological line bundle of $\P V$ (so that, for example, $\Gamma(\O(1),\P V) = V^{\vee}$)}
\item{$\mathcal{O}(m)$, the $m^{th}$ tensor power of $\mathcal{O}(1)$ over $\O_{\P V}$}
\item{$(m)$, the operation of tensor product of a sheaf of modules over $\O_{\P V}$ with $\mathcal{O}(m)$, also called a \emph{twist} by $m$}
\end{itemize}

In case the operand of $(m)$ is only a $k$ vector space, and not a sheaf, the operand is interpreted as a sheaf by tensor product over $k$ with $\mathcal{O}_{\P V}$. For example, $V(1)$ means the twist by $1$ of $V\otimes_{k}\mathcal{O}_{\P V}$. Note that the symbol $\otimes$ refers to tensor product over $\mathcal{O}$, wherever this interpretation is possible.

We generally use the Zariski topology for the spaces which are naturally algebraic varieties.

The groups, their representations, and their characters are generally algebraic.

%\subsection{Selected notation reference}
%
%The following notation is introduced in the text, and recorded here for reference.
%
%\begin{itemize}
%\item{$\L^p:=\L^pV^{\vee}$}
%\item{$\S^p:=\S^pV^{\vee}$}
%\item{$d$, the de Rham differential}
%\item{$\delta$ or $\iota_e$, the Koszul differential}
%\item{$\Omega^p$, the sheaf of algebraic differential $p$-forms on $\P V$}
%\item{$G:=\GL(V)$}
%\item{$[v_0]$, a chosen basepoint in $\P V$}
%\item{$H\subset G$, the stabilizer of $[v_0]$}
%\item{}
%\item{}
%\item{}
%\item{}
%\item{}
%\item{}
%\end{itemize}
\newpage

\subsection{Introduction}

The goal of this paper is to explain the visual-geometric phenomenon of multi-focal tensors in mathematical terms, and to import some of the ideas about them developed by computer/machine-vision scientists into mathematics.

The reader already familiar with the subject may wish to skip directly to the Main Construction in section \ref{mainconstructionsec}, the Main Application in section \ref{mainappsec}, and the section \ref{constraintsEucTri} concerning new constraints on the Euclidean trifocal variety.

It should be mentioned explicitly that the aspects which are reducible to computations in the Grassman algebra of meets and joins in 3 spatial dimensions are not the work of your present interlocutor, but rather that of several members of the computer vision research community, notably Faugeras, Hartley, Heyden, Luong, Papadopoulo, and Zisserman, among others, undertaken over the last couple of decades. The principal original contribution made here is a generalization in which the operations in the Grassman algebra are replaced with arbitrary linear skew-tensor invariants of the linear group $\GL(V)$, for a vector space $V$ of arbitrary finite dimension.

In section \ref{twodesc}, we relate skew-symmetric tensors over $V$ with certain functions defined on $GL(V)$ satisfying an equivariance property, by abstractly identifying both with weighted algebraic differential forms on $\P V$.\footnote{The idea to understand the elements of $\L^{p}V^{\vee}$ in terms of differential equations on $\P V$ actually goes back to Sophus Lie in 1877; see the exposition \cite{dolg} page 574, and \cite{liePfaff}.} An explicit isomorphism is also provided.

In section \ref{frsec} we explain how to regard our equivariant functions on $\GL(V)$ as ordinary functions, without special equivariance properties, defined on certain spaces of frames for tangent spaces of $\P V$. This is essentially the expression of the corresponding differential form with respect to these frames.

Up to this point the discussion is, informally speaking, only ``mono-focal". The first ``multi-focal" objects appear in section \ref{invformssec}, where we give the Main Construction:

Let $I$ be a joint relative $\GL(V)$-invariant of tuples of skew-symmetric tensors. That is, for some 1-dimensional character $\chi$ of $\GL(V)$,

\[I\in (\L^{p_1+1}V^{\vee}\otimes .. \otimes \L^{p_n+1}V^{\vee}\otimes \chi)^{\GL(V)}\]

To $I$ there corresponds:
\begin{enumerate}
\itemsep0em
\item{a $\GL(V)$-invariant weighted algebraic differential form on $(\P V)^{n}$, and}
\item{a $\GL(V)$-invariant algebraic function $f_w(I)$, defined on $n$-tuples of weighted frames in $\P V$, with values in the tensor space $\bigotimes_{i=1}^{n} \L^{p_i}T^{\vee}$.
}
\end{enumerate}
Here $T:=[v_0]^{\vee}\otimes V/[v_0]$ is the tangent space of $\P V$ at a basepoint $[v_0]$, regarded as an abstract tangent space of $\P V$.

In section \ref{mainappsec} we carry out this construction for some standard invariants in dimensions $\dim V = 2,3,4$. In the case $\dim V =4$ the varieties of ``multi-focal tensors" appear as the image varieties of the maps $f(I)$. In section \ref{geomconseqsec} some of the well-known properties of these varieties are explained by geometric properties of the corresponding invariant $I$.

In section \ref{practicalsec} we discuss why it is important to know more about the varieties $f(I)$ in the context of visual geometry, especially: equations cutting out $f(I)$ in the relevant tensor space, and certain integrability conditions analogous to the Maurer-Cartan equation. A proof is presented in section \ref{constraintssec} that certain new polynomial equations of degrees 4 and 5 hold on the tri-focal variety in the Euclidean setting.

\section{Invariant differential forms on powers of a projective space}\label{invforms}

\subsection{Two descriptions of algebraic differential forms on $\P V$}\label{twodesc}

\subsubsection{First description: Koszul cycles}

\begin{lemma} \label{eulers}(The Euler sequence)

There is a natural exact sequence of sheaves on $\P V$:

\[ 0\rightarrow k \rightarrow V(1)\rightarrow T_{\P V} \rightarrow 0,\]

where $T_{\P V}$ is the tangent sheaf of $\P V$.
\end{lemma}

\begin{proof}
See \cite{oko}, page 6.
\end{proof}

\begin{lemma} \label{rexact}(Right exactness of the exterior power)

Let $K\rightarrow A\rightarrow B\rightarrow 0$ be an exact sequence of modules over a commutative ring $R$.  Then the following induced sequences, for each $p\geq 1$, are exact:

\[ K\otimes_R \L^{p-1}_R(A)\rightarrow \L^{p}_R A \rightarrow \L^{p}_R B\rightarrow 0, \]

where the first map is the exterior product.
\end{lemma}

\begin{proof}
See \cite{ei}, pages 577-578.
\end{proof}

In the cases $p\geq 1$, apply Lemma \ref{rexact} to the $\O$-modules comprising the values of the sheaves in the sequence of Lemma \ref{eulers} (Note that, technically, Lemma \ref{rexact} applies to modules and not to sheaves). For $q\geq 0$, the $p+q$ twist of the $k$-linear dual of the result is an exact sequence

\begin{align*}
0\rightarrow \Omega^{p}(p+q)\rightarrow \L^{p}V^{\vee}(q)\rightarrow \L^{p-1}V^{\vee}(q+1)
\end{align*}

Denote the functor of global sections over $\P V$ by $H^{0}$. This functor is left exact, inducing an exact sequence

\begin{align}\label{kdef}
0\rightarrow H^{0}(\Omega^{p}(p+q))\rightarrow \L^{p}\otimes S^{q}\rightarrow\L^{p-1}\otimes S^{q+1}
\end{align}

where $\L^{p}:=\L^{p}V^{\vee}$ and $S^{q}:=\operatorname{Sym}^{q}V^{\vee}$.

Now set $\Omega_{Pol}:=\bigoplus_{p,q\geq 0} \L^{p}\otimes S^{q}$. In our context we use the following definition.

\begin{definition}
The \emph{Koszul differential}, denoted $\delta$, is the $(p,q)$ bidegree $(-1,1)$ graded mapping

\[ \delta: \Omega_{Pol}\rightarrow \Omega_{Pol}\]

equal to the direct sum of: the right-hand maps of the sequences (\ref{kdef}), and the zero maps on the summands $\L^{0}\otimes S^{q}$.
\end{definition}

The following proposition is a direct consquence.

\begin{proposition}
For every $p,q\geq 0$, $(p,q)\neq(0,0)$, there are natural isomorphisms:

\begin{align*}
H^{0}(\Omega^{p}(p+q), \P V)	& \cong	\ker (\delta:	\L^{p}		\otimes	S^{q}	\rightarrow	\L^{p-1}	\otimes S^{q+1}	)\\
\end{align*}
\end{proposition}

\begin{remark} The Koszul differential has the following combinatorial description in terms of decomposable elements.

\begin{align*}
\delta(\xi_1\xi_2\xi_3\xi_4..\xi_p\otimes f_1..f_q) = \,\,\,\xi_2\xi_3\xi_4..\xi_p&\otimes \xi_1\,f_1..f_p \\
- \xi_1\,\,\,\xi_3\xi_4..\xi_p&\otimes \xi_2\,f_1..f_p \\
+ \xi_1\xi_2\,\,\,\xi_4..\xi_p&\otimes \xi_3\,f_1..f_p \\
& \vdots
\end{align*}
\end{remark}

\begin{remark}
The Koszul differential also has the following differential-geometric description. $\Omega_{Pol}$ can be regarded as the algebra of polynomial differential forms on the algebraic manifold $V$. As an operation on $\Omega_{Pol}$, $\delta$ is the interior product with the vector field represented by the identity element $e\in V\otimes V^{\vee}$, known as the Euler field. This explains why $\delta$ is a differential, with $\delta^2 = 0$: the two-fold interior product with the vector field $e$ is the interior product with the vector field $e\wedge e = 0$.
\end{remark}

\begin{definition}
The \emph{de Rham differential} is the bidegree $(1,-1)$ graded mapping

\[d:\Omega_{Pol} \rightarrow \Omega_{Pol}\]

defined combinatorially in terms of decomposable elements as follows:

\begin{align*}
d(\xi_1..\xi_p\otimes f_1f_2f_3f_4..f_q) = \xi_1..\xi_p\,f_1&\otimes \,\,\,f_2f_3f_4..f_p \\
+  \xi_1..\xi_p\,f_2&\otimes f_1\,\,\,f_3f_4..f_p \\
+  \xi_1..\xi_p\,f_3&\otimes f_1f_2\,\,\,f_4..f_p \\
& \vdots
\end{align*}
\end{definition}

\begin{remark} Note that in the above, no signs arise (the $f_i$ commute).
\end{remark}

\begin{remark}
This de Rham differential $d$ agrees with the usual exterior derivative with respect to the geometric realization of $\Omega_{Pol}$, except for an overall sign $(-1)^{p}$ arising from the conventional use of $S^{q}\otimes \L^{p}$ rather than $\L^{p}\otimes S^{q}$.
\end{remark}
\begin{lemma} \label{char0}\hspace{1 mm}
\begin{enumerate}
\itemsep0em
\item{\label{cartan} $\delta$ and the de Rham differential $d$ satisfy
\[d\delta + \delta d = (p+q)\operatorname{id}_{S^{p}\otimes \L^{q}}\]}
\item{\label{exact} Except in grade $(p,q)=(0,0)$, the complexes $(\Omega_{Pol},\delta)$ and $(\Omega_{Pol},d)$ are exact.}
\item{\label{correspondence} Except in grade $(p,q)=(0,0)$, $d$ restricts to an isomorphism
\[ \ker \delta \cong \im d,\]
and $\delta$ restricts to an isomorphism
\[ \ker d \cong \im \delta\]
}
\end{enumerate}
\end{lemma}
\begin{proof} For the Cartan formula (\ref{cartan}), which supplies the chain homotopies that prove (\ref{exact}), the reader is referred to \cite{ws} page 5. The proof of (\ref{correspondence}) is a straightforward exercise; we only remark that this is where we use the assumption that the characteristic of the field $k$ is 0, in order to invert the integers $p+q$ in case $p+q\neq 0$.
\end{proof}
Let us concentrate on the case $q=1$.
\begin{corollary}\label{desc1} (First description of weighted algebraic differential forms on $\P V$)
There are natural isomorphisms:
\begin{alignat*}{12}
H^{0}(\Omega^{p}(p+1), \P V)	\cong	&\ker		&\delta:&	&\L^{p}	\otimes	S^{1}	&		&\rightarrow&	\L^{p-1}	\otimes S^{2}	&	 \\
								\cong	&\im \quad	&\delta:&	&\L^{p+1}	 				&		&\rightarrow&	\L^{p}	\otimes S^{1}		&	 \\
 								\cong	&\im  		&d:&		&\L^{p}	\otimes	S^{1}	&		&\rightarrow&	\L^{p+1}					&	 \\
 								\cong	&\ker		&d:&		&\L^{p+1}					&		&\rightarrow&	0 							&	 \\
\end{alignat*}
\end{corollary}
\begin{remark}
Although we will not need to use the following facts, we include them for the purpose of illustration.  It follows from the calculation of Schur polynomials (see \cite{fulton} for the details of such calculations) that $\L^{p}\otimes S^{q}$ is irreducible as a $\GL(V)$ representation if either $p=0$ or $q=0$, and it has exactly 2 non-isomorphic irreducible summands if both $p,q\geq 1$; in this case the Young diagrams labeling the summands appearing in $\L^{p}\otimes S^{q}$ are equal to the two diagrams which can be obtained by joining a $p$-tall column and a $q$-wide row along their first boxes. The resulting pairs of isomorphic summands are exchanged by $\delta$ and $d$, as illustrated below.
\end{remark}

\begin{figure}[H]
  \centering
   \includegraphics[scale=0.15]{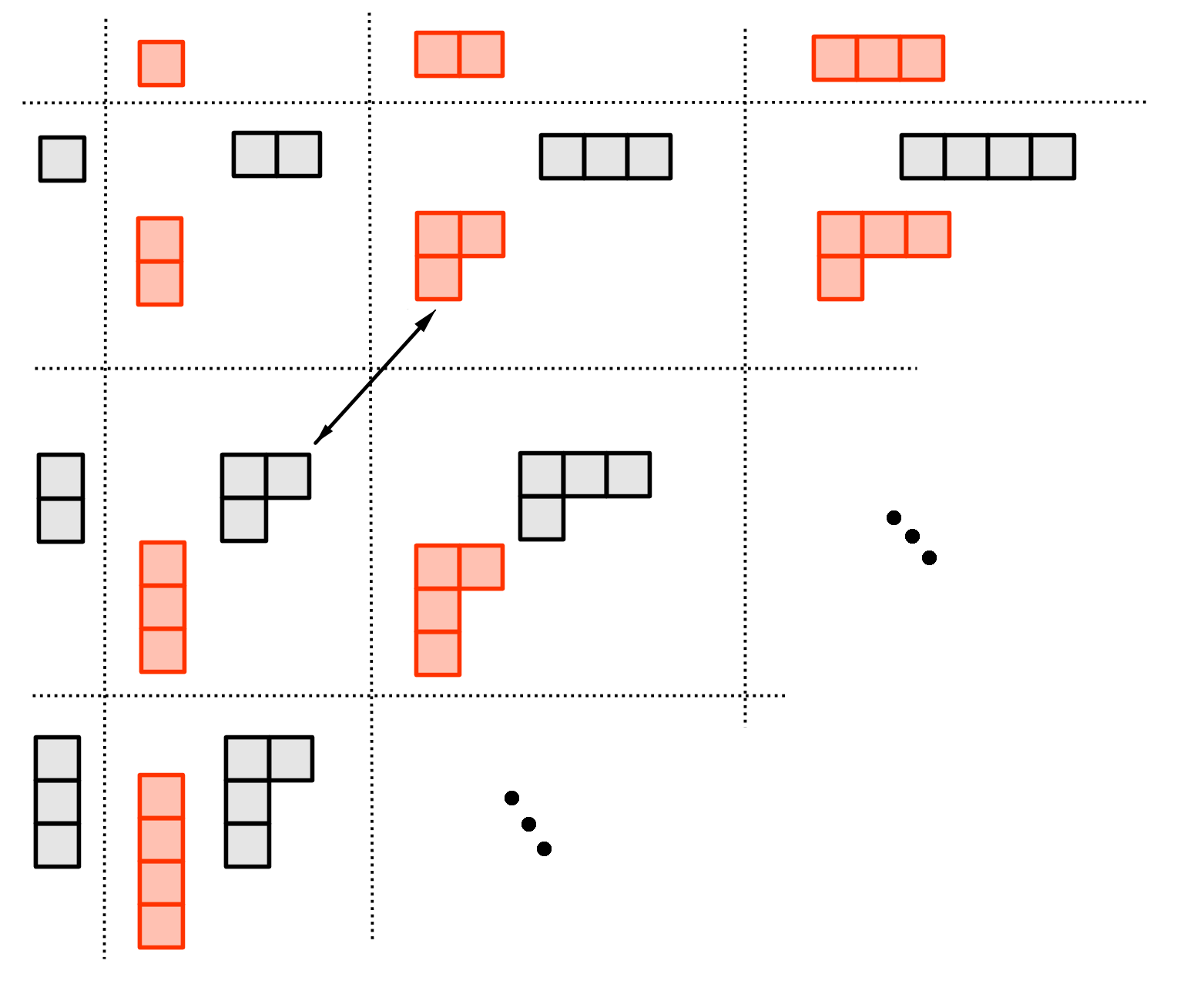}
   \caption{The Young diagrams of the summands of the $GL(V)$ irreducible decomposition of the polynomial de Rham complex $\Omega_{Pol}(V)$.}
\end{figure}

\subsubsection{Second description: Sections of associated bundles}
Now we adopt a different point of view on $\P V$, as a $\GL(V)$ homogeneous space with principal fiber bundle
\[ H\rightarrow \GL(V)\overset{\pi}{\rightarrow} \P V \]
The map $\pi$ is specified by $g\mapsto g\cdot [v_0]$ for some chosen basepoint $[v_0]\in \P V$, fixed from now on. $H$ denotes the stabilizer subgroup $\operatorname{stab}([v_0])\subset \GL(V)$.

More explicitly: Set $m=\dim V$ and extend $v_0$ to a basis for $V$. Let $g_{ij}$ denote the corresponding matrix coordinate functions on $\GL(V)$, in row $i$ and column $j$.
\[ \pi:
\begin{pmatrix}
g_{11}& .. & g_{1m}\\
\vdots& .. & \vdots \\
g_{m1} & .. & g_{mm}
\end{pmatrix} \rightarrow [g_{11},g_{12},..,g_{1m}]\]
The latter brackets $[..]$ denote homogeneous projective coordinates.

With respect to these coordinates, the subgroup $H$ is described by the equations
\[g_{12}=0\qquad g_{13}=0\qquad ...\qquad g_{1m}=0\] 
Let's also use the notation
\begin{align}\label{lwt}
\begin{split}
L&=[v_0]^{\vee}\\
W&=(V/[v_0])^{\vee}\\
T&=L\otimes W^{\vee}
\end{split}
\end{align}
so that for example $L^{q}\otimes \L^{p}W = L^{p+q}\otimes\L^{p}T^{\vee}$.
\begin{definition} Let $A$ be any group. If $X$ is a right $A$-set and $Y$ is a left $A$-set, the \emph{balanced product} $X\underset{bal}{\times}Y$ is the right $A$-set $X\times Y^{op}$ with the diagonal action. Here $Y^{op}$ denotes the opposite right $A$-space, with action $y \cdot a := a^{-1} \cdot y$.

Similarly, if $X$ and $Y$ are vector spaces and the $A$ actions are linear, $X\underset{bal}{\otimes}Y$ is the ordinary tensor product $X\otimes Y^{op}$, with the diagonal right action of $A$.
\end{definition}
\begin{remark} (Terminological ambiguity). If $X$ is only a right $A$-space and $Y$ is only a left $A$-space, there is only one way to interpret $X\times Y$ as a right $A$ space in such a way that both actions are used. Namely, as the balanced product defined above. Therefore in these cases we will omit the subscript notation $_{bal}$.
\end{remark}
\begin{definition} Let $A$ be any subgroup of a group $B$. Consider the quotient projection $\pi:B\rightarrow B/A$.  If $M$ is a representation of $A$, the bundle \emph{$\pi$-associated to $M$} is the quotient set $(B\times M)/A$, with its projection to $B/A$.  It may be denoted $B\times_{A} M$.
\end{definition}
\begin{remark}
Since the left action of $B$ on itself and the right action of $A$ on $B$ commute, the bundle $\pi$-associated to $M$ inherits a left $B$ action.
\end{remark}
Typically, but not always, if $A$ and $B$ are group objects in a category like topological spaces, manifolds, algebraic varieties, etc., then $\pi$, the $\pi$-associated bundles, and their $B$ actions also belong to this category.
\begin{proposition} \emph{(Associated bundles)} There are canonical $\GL(V)$-equivariant isomorphisms of algebraic varieties,
\begin{enumerate}
\itemsep0em
\item{between the line bundle $\O_{\P V}(1)$ and the bundle $\pi$-associated to the $H$ representation $L$,}
\item{between the tangent bundle of $\P V$ and the bundle $\pi$-associated to the $H$ representation $\mathfrak{gl}(V)/\mathfrak{h}\cong T$, and}
\item{between the vector bundle $\Omega^{p}(p+q)$ and the bundle $\pi$-associated to the $H$ representation $L^{q} \otimes \L^{p}W$.}
\end{enumerate}
\end{proposition}
\begin{proof}
See \cite{capslovak}.
\end{proof}
\begin{proposition} \emph{(Second description of algebraic differential forms on $\P V$)}\label{desc2}
\begin{align*}
 H^{0}(\Omega^{p}(p+q),\P V)	\cong&	\operatorname{Maps}_H(\GL(V), L^{q} \otimes \L^{p}W)\\
 								\cong&	(\O_{\GL(V)}\underset{bal}{\otimes} L^{q} \otimes \L^{p}W)^{H}
\end{align*}
where $\operatorname{Maps}_H(\bullet, \bullet)$ denotes the set of $H$-equivariant morphisms of affine varieties, and the superscript $^H$ denotes the subspace of $H$-invariant elements.
\end{proposition}
\begin{proof}  The first isomorphism is the standard description of the set of sections of an associated bundle (see \cite{capslovak}).  The set of algebraic maps between affine varieties over $k$ is naturally isomorphic to the set of $k$-algebra homomorphisms, mapping the opposite direction, between the algebras of regular functions (\cite{hart}, page 19).  In turn these are specified by prescribing arbitrarily their values on generators.  In this case, $\O_{L^{q} \otimes \L^{p}W}$ is generated by $ L^{q\vee} \otimes \L^{p}W^{\vee}$, so that the space of mappings is the displayed tensor product.  The condition of $H$-equivariance for the maps $\GL(V)\rightarrow L^{q} \otimes \L^{p}W$ is equivalent to $H$-invariance of the corresponding tensor, with the balanced tensor product $H$ action.
\end{proof}
\subsubsection{Explicit relation between the two descriptions}
Now define maps $\psi: \L^{p+1}V^{\vee} \rightarrow \O_{\GL(V)}\underset{bal}{\otimes} L \otimes \L^{p}W$ as follows.

The effect of the action $\GL(V)\times \L^{p+1}V\rightarrow \L^{p+1}V$ by pullback on regular functions is the so-called coaction
\[ \O_{\L^{p+1}V} \rightarrow \O_{\GL(V)}\otimes\O_{\L^{p+1}V} \]
Since $\L^{p+1}V$ is a linear representation of $\GL(V)$, the image of the restriction of the coaction to the linear forms $\L^{p+1}V^{\vee}\subset \O_{\L^{p+1}V}$ is contained in the linear forms:
\[ \L^{p+1}V^{\vee} \rightarrow \O_{\GL(V)}\otimes \L^{p+1}V^{\vee} \]
We define $\psi$ to be the composition of this map with: the identity of $\O_{GL(V)}$, tensor with the $k$-linear dual of the exterior product map shown below
\begin{align*}
 [v_0] \otimes \L^{p}(V/[v_0]) &\rightarrow \L^{p+1}V \\
 L^{\vee}\otimes \L^{p}W^{\vee} &\rightarrow \L^{p+1}V
\end{align*}
Summarizing:
\begin{definition}\label{psimaps}
\label{psi} For each $p\geq 0$, define a map $\psi=\psi_p$ as the composition:
\[\psi: \L^{p+1}V^{\vee}\rightarrow \O_{\GL(V)} \underset{bal}{\otimes} \L^{p+1}V^{\vee}\rightarrow \O_{\GL(V)} \underset{bal}{\otimes} L\otimes \L^{p}W\]
\end{definition}
\begin{remark}
We shall generally suppress the subscript $_p$ indicating which of the maps $\psi$ is being used. This way the decorations of $\psi$ can be used instead to indicate tensor indices.
\end{remark}
\begin{proposition} \label{idpsi}\hspace{1 mm} % name?
\begin{enumerate}
\itemsep0em
\item{The image of the map $\psi$ consists of $H$-invariants. \label{hinv}}
\item{The map $\psi$ is proportional to the isomorphism of $GL(V)$ modules equal to the composition of the descriptions of Corollary \ref{desc1} and of Proposition \ref{desc2}\label{comp}:
\begin{align*}
 \L^{p+1}V^{\vee} &\cong H^{0}(\Omega^{p}(p+1),\P V) \cong (\O_{\GL(V)}\otimes L \otimes \L^{p}W)^{H}
\end{align*}}
\end{enumerate}
\end{proposition}
 The proof requires no insight beyond correct application of the definitions. The reader may wish to skip it or to supply it him or herself.
\begin{proof}(\ref{hinv}) Temporarily denote by $G$ an arbitrary group. For a general $G$-space $X$, in any category of ringed spaces for which $\O_{G\times X}\cong \O_G\otimes\O_X$, the image of the coaction $\O_X\rightarrow \O_G \otimes \O_X$ is contained in the subspace of $G$-invariant elements of the balanced tensor product $\O_G \otimes \O_X$.  This is a direct consequence of the following form of the defining property of an action:
\[(ab^{-1})\cdot (b\cdot x)=a\cdot x \qquad a,b\in G\quad x\in X \]
Thus the image of the first factor comprising the map $\psi$ in Definition \ref{psimaps} lies in the $\GL(V)$-invariants of the balanced tensor product $\O_{\GL(V)}\otimes \L^{p+1}V^{\vee}$.

The identity map $\O_{\GL(V)}\rightarrow \O_{\GL(V)}$ is certainly a mapping of right $H$-modules, indeed even of $\GL(V)$-bimodules, while the map 
\[ \L^{p+1}V^{\vee} \rightarrow L\otimes \L^{p}W\]
is only a map of left $H$-modules (since the target is only an $H$-module).  Nevertheless, as a result, the map
\[\O_{\GL(V)}\otimes \L^{p+1}V^{\vee} \rightarrow \O_{\GL(V)} \otimes L\otimes \L^{p}W\]
is balanced-$H$-equivariant.  $H$-invariants must map to $H$-invariants under such a map, so the image of $\psi$ consists of $H$-invariants.

(\ref{comp}) Now consider $\psi$ as a mapping of ordinary tensor products of left $\GL(V)$-spaces,
\[\L^{p+1}V^{\vee}\rightarrow\O_{\GL(V)} \otimes \L^{p+1}V^{\vee} \rightarrow \O_{\GL(V)} \otimes L\otimes\L^{p}W\]
The following form of the defining property of an action,
\[ a\cdot (b\cdot x) = (ab)\cdot x,\] 
implies that the first factor of the map $\psi$ is left $GL(V)$-equivariant for
\begin{itemize}
\itemsep-0.2em
\item{the domain equipped with the usual left action on $\L^{p+1}V^{\vee}$, and}
\item{for the target equipped with the left action on $\O_{\GL (V)}$ tensor the trivial left action on $\L^{p+1}V^{\vee}$.}
\end{itemize}
Thus $\psi$ is left $GL(V)$-equivariant with respect to
\begin{itemize}
\itemsep-0.2em
\item{the usual left action on the domain $\L^{p+1}V^{\vee}$, and}
\item{the left action on $\O_{\GL(V)}$ tensor the trivial left action on $L\otimes\L^{p}W$.}
\end{itemize}
Since $\psi$ and the displayed composite isomorphism are both $\GL(V)$ equivariant and $\L^{p+1}V^{\vee}$ is irreducible, by Schur's lemma these maps are proportional.
\end{proof}
\subsubsection{Basis expression for $\psi$}\label{basisexpr}
We will need to describe the maps $\psi$ very explicitly. Set $m=\dim V$ and let $v_0, \dots, v_{m-1} \in V$ denote an arbitrary extension of the basepoint $v_0$ to a basis for $V$. Let $e_0, \dots, e_{m-1} \in V^{\vee}$ be the dual basis. Note that there is an injection $(V/[v_0])^{\vee}\subset V^{\vee}$ and a projection $V^{\vee} \rightarrow [v_0]^{\vee}$. Let
\[\tilde{e}_1,\tilde{e}_2,\dots,\tilde{e}_{m-1} \in (V/[v_0])^{\vee}\]
be the basis corresponding to $e_1,e_2,\dots,e_{m-1}$ under the injection, and let
\[\tilde{e}_0 \in [v_0]^{\vee}\]
be the basis element equal to the image of $e_0$ under the projection.

Let $g_{ij}$ denote the matrix coordinate functions of $\GL(V)\cong \GL(m)$, where $i$ specifies row and $j$ specifies column. The indices $i$ and $j$ run from $0$ to $m-1$.

Denote the first map comprising $\psi$, the coaction, by $\mu : \L^{p+1}V^{\vee}\rightarrow \mathcal{O}_{\GL(V)}\otimes \L^{p+1}V^{\vee}$.
\begin{example} ($V\cong k^{2}, p=0$) \label{k2p0}
\[\psi: \L^{1}k^{2\vee}\rightarrow \mathcal{O}_{\GL(2)}\otimes [v_0]^{\vee}\otimes\L^{0}(k^2/[v_0])^{\vee}\]
Although the cases where $p=0$ are rather degenerate, we include some of them for completeness.

The first factor of $\psi$:
\begin{align*}
& \mu(e_0) = g_{00} \otimes e_0 + g_{01}\otimes e_1\\
& \mu(e_1) = g_{10} \otimes e_0 + g_{11}\otimes e_1
\end{align*}
The second map defining $\psi$ is the interior product with $v_0$ (followed by the formal tensor product with $\tilde{e}_0$):
\begin{align*}
& \psi(e_0) = g_{00}\otimes \tilde{e}_0 \otimes 1 + 0\\
& \psi(e_1) = g_{10}\otimes \tilde{e}_0 \otimes 1 + 0
\end{align*}
\end{example}
\begin{example} ($V\cong k^{3}, p=0$) \label{k3p0}
\[\psi: \L^{1}k^{3\vee}\rightarrow \mathcal{O}_{\GL(3)}\otimes [v_0]^{\vee}\otimes\L^{0}(k^3/[v_0])^{\vee}\]
\begin{align*}
& \mu(e_0) = g_{00} \otimes e_0 + g_{01}\otimes e_1 + g_{02}\otimes e_2\\
& \mu(e_1) = g_{10} \otimes e_0 + g_{11}\otimes e_1 + g_{12}\otimes e_2\\
& \mu(e_2) = g_{20} \otimes e_0 + g_{21}\otimes e_1 + g_{22}\otimes e_2
\end{align*}
\begin{align*}
& \psi(e_0) = g_{00}\otimes \tilde{e}_0 \otimes 1 + 0 + 0\\
& \psi(e_1) = g_{10}\otimes \tilde{e}_0 \otimes 1 + 0 + 0\\
& \psi(e_2) = g_{20}\otimes \tilde{e}_0 \otimes 1 + 0 + 0
\end{align*}
\end{example}
\begin{example} ($V\cong k^{3}, p=1$)\label{k3p1}
\[\psi: \L^{2}k^{3\vee}\rightarrow \mathcal{O}_{\GL(3)}\otimes [v_0]^{\vee}\otimes\L^{1}(k^3/[v_0])^{\vee}\]
Use the basis
\[e_0e_1\quad e_0e_2\quad e_1e_2\in\L^{2}k^{3\vee}\]
(We are suppressing the explicit notation for the exterior product $\wedge$.)
\begin{align*}
& \mu(e_0e_1) =(g_{00}g_{11}-g_{01}g_{10})\otimes e_0e_1 + (g_{00}g_{12}-g_{02}g_{10})\otimes e_0e_2 + (g_{01}g_{12}-g_{02}g_{11})\otimes e_1e_2\\
& \mu(e_0e_2) =(g_{00}g_{21}-g_{01}g_{20})\otimes e_0e_1 + (g_{00}g_{22}-g_{02}g_{20})\otimes e_0e_2 + (g_{01}g_{22}-g_{02}g_{21})\otimes e_1e_2  \\
& \mu(e_1e_2) =(g_{10}g_{21}-g_{11}g_{20})\otimes e_0e_1 + (g_{10}g_{22}-g_{12}g_{20})\otimes e_0e_2 + (g_{11}g_{22}-g_{12}g_{21})\otimes e_1e_2  \\
\end{align*}
\vspace{-3pc}
\begin{align*}
& \psi(e_0e_1)= (g_{00}g_{11}-g_{01}g_{10})\otimes \tilde{e}_0\otimes \tilde{e}_1 + (g_{00}g_{12}-g_{02}g_{10})\otimes \tilde{e}_0\otimes\tilde{e}_2 +0 \\
& \psi(e_0e_2)=  (g_{00}g_{21}-g_{01}g_{20})\otimes \tilde{e}_0\otimes \tilde{e}_1 + (g_{00}g_{22}-g_{02}g_{20})\otimes \tilde{e}_0\otimes \tilde{e}_2 + 0 \\
& \psi(e_1e_2)=  (g_{10}g_{21}-g_{11}g_{20})\otimes\tilde{e}_0\otimes \tilde{e}_1 + (g_{10}g_{22}-g_{12}g_{20})\otimes \tilde{e}_0\otimes \tilde{e}_2 + 0\\
\end{align*}
To simplify notation, use the bases to regard the target as a free $\O_{\GL(4)}$ module of rank $\dim_k(\L^{1}(k^3/[v_0])^{\vee})=2$, and express the values of $\psi$ as row vectors with entries in $\O_{\GL(4)}$. Also denote the values of $\psi$ on the basis elements by $\psi_{ij}=\psi(e_ie_j)$, and adopt the notation $g^{ij}_{kl}$ for the minor of the matrix $[g]$ at columns $i,j$ and rows $k,l$:
\begin{align}\label{firstnotationmin}
\begin{split}
&[\psi_{01}]= \begin{bmatrix} g^{01}_{01} & g^{02}_{01}\end{bmatrix}\\
&[\psi_{02}]= \begin{bmatrix} g^{01}_{02} & g^{02}_{02}\end{bmatrix}\\
&[\psi_{12}]= \begin{bmatrix} g^{01}_{12} & g^{02}_{12}\end{bmatrix}
\end{split}
\end{align}
\end{example}
\begin{example} \label{firstDim4}\label{k4p1} ($V\cong k^{4}, p=1$)
\[\psi: \L^{2}k^{4\vee}\rightarrow \mathcal{O}_{\GL(4)}\otimes [v_0]^{\vee}\otimes\L^{1}(k^4/[v_0])^{\vee}\]
% Consider the case that $\dim V = 4$ and $p=1$.
% Let $v_0,v_1,v_2,v_3\in V$ be an extension of the chosen element $v_0$ to a basis for $V$, and let $e_0,e_1,e_2,e_3\in V^{\vee}$ be the dual basis.
Use the bases
\[e_0e_1\quad e_0e_2\quad e_0e_3\quad e_1e_2\quad e_1e_3\quad e_2e_3\in\L^{2}k^{4\vee}\]
% Note that there is an inclusion $(V/[v_0])^{\vee}\subset V^{\vee}$, with induced basis
% \[ e_1,e_2,e_3 \in \L^{1}(V/[v_0])^{\vee} \]
% We also select the obvious basis $v_0 \in [v_0]$, and denote its dual by $v_0^{\vee}$. Note that $e_0$ corresponds to $v_0^{\vee}$ under restriction $V^{\vee}\rightarrow [v_0]^{\vee}$, but $e_0$ and $v_0^{\vee}$ are not literally equal.
% Letting $g_{ij}\in \O_{\GL(V)}$ denote the matrix coordinate functions with respect to these bases ($i,j$ from $0$ to $3$),
\[ \tilde{e}_0\otimes \tilde{e}_1 \quad \tilde{e}_0\otimes \tilde{e}_2 \quad \tilde{e}_0\otimes \tilde{e}_3 \in [v_0]^{\vee}\otimes(k^{4}/[v_0])^{\vee}\]
Using the notation introduced for item (\ref{firstnotationmin}),
\begin{align*}
[\psi_{01}] = \begin{bmatrix} g^{01}_{01} & g^{02}_{01} & g^{03}_{01} \end{bmatrix}\\
[\psi_{02}] = \begin{bmatrix} g^{01}_{02} & g^{02}_{02} & g^{03}_{02} \end{bmatrix}\\
[\psi_{03}] = \begin{bmatrix} g^{01}_{03} & g^{02}_{03} & g^{03}_{03} \end{bmatrix}\\
[\psi_{12}] = \begin{bmatrix} g^{01}_{12} & g^{02}_{12} & g^{03}_{12} \end{bmatrix}\\
[\psi_{13}] = \begin{bmatrix} g^{01}_{13} & g^{02}_{13} & g^{03}_{13} \end{bmatrix}\\
[\psi_{23}] = \begin{bmatrix} g^{01}_{23} & g^{02}_{23} & g^{03}_{23} \end{bmatrix}
\end{align*}
\end{example}
\begin{example} \label{secondDim4} \label{k4p2}($V\cong k^{4}, p=2$)
\[\psi: \L^{3}k^{4\vee}\rightarrow \mathcal{O}_{\GL(4)}\otimes [v_0]^{\vee}\otimes\L^{2}(k^4/[v_0])^{\vee}\]
% In the setting of the previous example, consider $p=2$ rather than $p=1$.
Use the bases
\[ e_0e_1e_2\quad	e_0e_1e_3\quad	e_0e_2e_3\quad	e_1e_2e_3 \in \L^{3}V^{\vee} \]
\[	\tilde{e}_0\otimes \tilde{e}_1\tilde{e}_2\quad	\tilde{e}_0\otimes \tilde{e}_1\tilde{e}_3\quad	\tilde{e}_0\otimes \tilde{e}_2\tilde{e}_3\quad	\in \Lambda^{2}(V/[v_0])^{\vee}\]
% Adopt the abbreviation $g^{ijk}_{lmn}$ to mean the minor of $g$ at columns $i,j,k$ and rows $l,m,n$.  We find that
\begin{align*}
&[\psi_{012}]=	\begin{bmatrix}	g^{012}_{012} & g^{013}_{012} & g^{023}_{012} \end{bmatrix}\\
&[\psi_{013}]=	\begin{bmatrix}	g^{012}_{013} & g^{013}_{013} & g^{023}_{013} \end{bmatrix}\\
&[\psi_{023}]=	\begin{bmatrix}	g^{012}_{023} & g^{013}_{023} & g^{023}_{023} \end{bmatrix}\\
&[\psi_{123}]=	\begin{bmatrix}	g^{012}_{123} & g^{013}_{123} & g^{023}_{123} \end{bmatrix}\\
\end{align*}
\end{example}
\subsection{Frames}\label{frsec}
\subsubsection{Frame spaces as intermediate quotients}
We elaborate on the homogeneous space structure of $\P V$.
\begin{definition} For a representation $M$ of a group $H$:
\begin{itemize}
\itemsep0em
\item{$\ker(M)\subset H$ denotes the elements of $H$ which act trivially on $M$.}
\item{$\widetilde{\ker}(M)\subset H$ denotes the elements of $H$ which act by scalar operators on $M$.}
\end{itemize}
\end{definition}
\begin{definition} Let us use the following notation:
$G=\GL(V)$, $H=\stab([v_0]\in \P V)\subset G$ (as before), and
\begin{align*}
K_L&=\ker(L)\subset H&	&\widetilde{K}_L=\widetilde{\ker}(L)\subset H\\
K_T&=\ker(T)\subset H&	&\widetilde{K}_T=\widetilde{\ker}(T)\subset H\\
K_W&=\ker(W)\subset H&	&\widetilde{K}_W=\widetilde{\ker}(W)\subset H
\end{align*}
\end{definition}
Here $L$,$T$, and $W$ are defined by (\ref{lwt}).
\begin{proposition} 
\[ \widetilde{K}_L = H \qquad K_T \subset \widetilde{K}_T \qquad K_W \subset \widetilde{K}_W\]
\end{proposition}
\begin{proof}
These statements are a matter of notation.  
\end{proof}
\begin{proposition} \label{twooutofthreeP}\hspace{1mm}
\begin{enumerate}
\itemsep0em
\item{$ \widetilde{K}_T = \widetilde{K}_W$}
\item{\label{twooutofthree}The intersection of any 2 of $K_L,K_T,K_W$ is the 3-fold intersection $K_{LTW}:=K_L\cap K_T \cap K_W$.}
\end{enumerate}
\end{proposition}
\begin{proof} 
These statements follow from the equations
\[ T = L\otimes W^{\vee} \qquad W= L\otimes T^{\vee}\] 
\end{proof}
On account of this proposition, set $\widetilde{K}_{T,W} := \widetilde{K}_T = \widetilde{K}_W$.
\begin{definition} \hspace{1mm}
Set $F=G/K_{T}$, and call this the space of \emph{frames} in $\P V$.
Set $F_w=G/K_{LTW}$, and call this the space of \emph{weighted frames} in $\P V$.
Set $F_{t}=G/K_{W}$, and call this the space of \emph{twisted (unweighted) frames} in $\P V$.
Set $F_s=G/\widetilde{K}_{T,W}$, and call this the space of \emph{frames up to scale} in $\P V$.
\end{definition}
\begin{remarkN}\label{framesisochoice} It will be convenient to make a choice of isomorphism $i:k^{m}\rightarrow T$, $i\in\operatorname{Isom}(k^{m},T)$, for the rest of this paper. The choice of $i$ is not canonical.\end{remarkN}
\begin{proposition} \label{framesiso}\hspace{1mm}
\begin{enumerate}
\itemsep0em
\item{\label{associatedframes2}There are $G$-equivariant isomorphisms between the three spaces:
\begin{itemize}
\item{$F$,}
\item{the bundle over $\P V$ $\pi$-associated to the $H$-space $\operatorname{Isom}(k^{m},T)$ (the space of ordered bases for $T$), and}
\item{the space of frames for tangent spaces of $\P V$.}
\end{itemize}}
\item{\label{scaledframes}There is a $G$-equivariant isomorphism $F_s \cong F/k^{\times}$ (where the action of $k^{\times}$ on $F$ is by rescaling).}
\end{enumerate}
\end{proposition}
\begin{proof}
(\ref{associatedframes2}) $G$ is transitive on all three spaces. The quotient $F=G/K_{T}$ has the basepoint equal to the class of the identity element of $G$.The bundle $\pi$-associated to $\operatorname{Isom}(k^{m},T)$ is endowed with the basepoint equal to the $H$-equivalence class of the pair $(1\in G,i)$. $T$ is naturally isomorphic with $T_{[v_0]}\P V$, so that $i$ also distinguishes a frame in $\P V$ at the chosen basepoint $[v_0]$. These basepoints have the same stabilizer in $G$, namely $K_{T}$. Thus the orbit-stabilizer construction provides the mutual isomorphism between these 3 spaces.

(\ref{scaledframes}) The group $K_T$ is normal in $\widetilde{K}_T$, with quotient isomorphic to $k^{\times}$. The extension is split by the subgroup of $\widetilde{K}_T$ equal to the scalar transformations.  Since this subgroup is central, in fact
\[\widetilde{K}_T \cong K_T \times k^{\times}\]
Thus
\[ F_s = G/\widetilde{K}_{T,W} = G/\widetilde{K}_{T} = G/(K_T \times k^{\times} ) = (G/K_T)/k^{\times} = F/k^{\times}\]
\end{proof}
\subsubsection{The subgroup lattices}
For the reader's convenience in organizing what has been said so far, we show: the lattice of the groups introduced in section \ref{frsec}, the lattice of the corresponding $G$-homogeneous quotients, and the restrictions of the $H$ representation $M=L^{p+q}\otimes \L^{p}W=L^{q}\otimes \L^{p}T^{\vee}$ to these subgroups. Factors enclosed by square brackets are [trivial], and factors enclosed by parantheses are trivial only up to (scale).
\[
\xymatrix{
			&H								&			& &			\\
			&\widetilde{K}_{T,W}\ar@{^{(}->}[u]				&			& &			\\
K_{T}\ar@{^{(}->}[ur]	&								&K_{W}\ar@{^{(}->}[ul]	& &K_{L}\ar@{^{(}->}[uulll]\\
			&K_{LTW}\ar@{^{(}->}[ul]\ar@{^{(}->}[ur]\ar@{^{(}->}[urrr]	&			& &			\\
}\]

\[
\xymatrix{
		&\P V				&		& &				\\
		&F_s\ar@{->>}[u]			&		& &				\\
F\ar@{->>}[ur]	&				&F_t\ar@{->>}[ul]	& &V\backslash\{0\}\ar@{->>}[uulll]	\\
		&F_w\ar@{->>}[ul]\ar@{->>}[ur]\ar@{->>}[urrr]	&		& &				\\
}\]

\[
\xymatrix{
			&L^{p+q}\otimes \L^{p}W=L^{q}\otimes \L^{p}T^{\vee}		&			& 				\\
			&L^{p+q}\otimes(\L^{p}T^{\vee})=L^{q}\otimes (\L^{p}W)	&			& 				\\
L^{p+q}\otimes[\L^{p}T^{\vee}]	&							&L^{q}\otimes[\L^{p}W]	& [L^{q}]\otimes \L^{p}W	\\
			&[L^{p+q}\otimes \L^{p}W=L^{q}\otimes \L^{p}T^{\vee}]	&			& 				\\
}\]
\subsubsection{Functions on frames spaces}
In this section we relate the mappings of Proposition \ref{desc2} to functions on $F_w$ and $F$.
\begin{proposition} \label{converttoframes}\hspace{1mm} 
For each $p,q\geq 0$, $(p,q)\neq (0,0)$, there is an injection of $\GL(V)$ modules
\begin{align*}
\operatorname{Maps}_H(GL(V),L^{q}\otimes\L^{p}W) \rightarrow & \operatorname{Maps}(F_w,L^{p+q}\otimes L^{p}T^{\vee})\\
& \cong H^{0}( \O_{F_w}, F_w)\otimes L^{p+q}\otimes\L^{p}T^{\vee}
\end{align*}
where the $GL(V)$ action on the right-hand side is induced by the left action on $F_w$.
\end{proposition}
%\begin{proof} An $H$-invariant is always also $K_{LTW}$-invariant; there is an injection
%\[(\O_{\GL(V)}\underset{bal}{\otimes} L^{q} \otimes \L^{p}W)^{H}\subset (\O_{\GL(V)}\underset{bal}{\otimes} L^{q} \otimes \L^{p}W)^{K_{LTW}}\]
%Since $L^{q}\otimes \L^{p}W=L^{p+q}\otimes\L^{p}T^{\vee}$ is trivial as a $K_{LTW}$ representation, the latter is just the space of algebraic maps $ \operatorname{Maps}(F_w,L^{p+q}\otimes L^{p}T^{\vee}) $.
\begin{proposition} \label{converttoframesunweighted}\hspace{1mm} 
For each $p,q\geq 0$, $(p,q)\neq (0,0)$, there is an injection of $\GL(V)$ modules
\begin{align*}
\operatorname{Maps}_H(GL(V),L^{q}\otimes\L^{p}W) \rightarrow H^{0}( \O_{F}(p+q), F)\otimes \L^{p}T^{\vee}
\end{align*}
\end{proposition}
%\begin{proof} Recall that $F=G/K_T$.  We only need to explain the notation: $\O_{F}(p+q)$ denotes the sheaf-of-modules pullback of $\O_{\P V}(p+q)$ from $\P V$ to $F$, or equivalently the sheaf of sections of the line bundle associated to the $K_T$ representation $L^{p+q}$.  The argument is the same as in the proof of Proposition \ref{converttoframes}.
The proofs of Propositions \ref{converttoframes} and \ref{converttoframesunweighted} follow directly from the definitions. We only need to explain the notation : $\O_{F}(p+q)$ denotes the sheaf-of-modules pullback of $\O_{\P V}(p+q)$ from $\P V$ to $F$, or equivalently the sheaf of sections of the line bundle associated to the $K_T$ representation $L^{p+q}$.
\subsection{Invariant forms}\label{invformssec}\label{mainconstructionsec}
\subsubsection{Main Construction}
\begin{theorem} \label{mainconstruction}
To each relative $\GL(V)$-invariant
%\footnote{Following convention, a ``$\GL(V)$-invariant in a representation $M$" really means an invariant element of $M\otimes k_\chi$, where $\chi$ is a character of $\GL(V)$ of the form $g\mapsto (\operatorname{det}g)^{c}$.}
\[I\in \L^{p_1+1}V^{\vee}\otimes .. \otimes \L^{p_n+1}V^{\vee}\]
there corresponds:
\begin{enumerate}
\itemsep0em
\item{a $\GL(V)$-invariant weighted differential form
\begin{align*}
\eta=\eta(I)\in &H^{0}(\Omega^{p_1}\boxtimes..\boxtimes\Omega^{p_n}(p_1+1,.., p_n+1), (\P V)^n)\\
 \subset &H^{0}(\Omega^{p_1+..+p_n}(p_1+1,.., p_n+1), (\P V)^n)
\end{align*}}
\item{a $\GL(V)$-invariant tensor
\[\tau=\tau(I)\in\left[
 \bigotimes_{i=1}
 \O_{\GL(V)}\underset{bal}{\otimes} L \otimes \L^{p_i}W
 \right]^{H^n}\]
}
\item{a $\GL(V)$-invariant regular function
\begin{align*}
f_w=f_w(I)	:F_w^{n} &\rightarrow \bigotimes_{i=1}^{n}L\otimes \L^{p_i}W
\end{align*}}
\item{$\GL(V)$-invariant rational functions
\begin{align*}
 f=f(I):F^{n} &\dashrightarrow \P \left(\bigotimes_{i=1}^{n} \L^{p_i}T^{\vee} \right)\\
 f_s=f_s(I):(F/k^{\times})^{n} &\dashrightarrow \P \left(\bigotimes_{i=1}^{n} \L^{p_i}T^{\vee}\right) 
\end{align*}}
\end{enumerate}
\end{theorem}
\begin{proof} To obtain $\tau(I)$, apply to $I$ the tensor product of the maps $\psi_{p_1},..,\psi_{p_n}$ (Definition \ref{psi}).
To obtain $\eta(I)$, apply Proposition \ref{idpsi} to $\tau(I)$ and then the Kunneth isomorphism
\begin{align*}
		& H^{0}(\Omega^{p_1}(p_1+1),\P V)\otimes .. \otimes H^{0}(\Omega^{p_n}(p_n+1),\P V)\\	
\cong	& H^{0}(\Omega^{p_1}(p_1+1)\boxtimes .. \boxtimes \Omega^{p_n}(p_n+1),(\P V)^{n})
\end{align*}
To obtain $f_w(I)$, apply Proposition \ref{converttoframes} to $\tau(I)$ and then the Kunneth isomorphism
\begin{align*}
	& H^{0}(\O_{F_w},F_w)\otimes .. \otimes H^{0}(\O_{F_w},F_w)\\	
\cong	& H^{0}(\O_{F_w}\boxtimes .. \boxtimes \O_{F_w},F_w^{n}) \\
\cong	& H^{0}(\O_{F_w^{n}},F_w^{n})
\end{align*}	%Is A(X x Y) = A(X) x A(Y) valid enough for this case?
To obtain $f(I)$: Regard $\tau(I)$ as an $H^{n}$-equivariant map
\[h: (\GL(V))^{n}\rightarrow \bigotimes_{i=1}^{n}L\otimes \L^{p_i}W\]
It descends to a rational map, still $H^{n}$-equivariant, defined away from $h^{-1}(0)$:
\[(\GL(V))^{n}\dashrightarrow \P(\bigotimes_{i=1}^{n}\L^{p_i}W) \]
We regard it as merely $K_T^{n}$-equivariant. Since the target is actually $K_T^{n}$-invariant, the map corresponds to a rational section of a trivial bundle over $F^{n}$, the map $f(I)$.

To obtain $f_s(I)$, repeat the procedure just described, which produced $f(I)$ from the group $K_T^{n}$, with the group $\widetilde{K}_T^{n}$ instead.
\end{proof}
\begin{remarkN} \label{termMF}(Terminology). The image of the map $f_w(I)$ in the tensor space $\bigotimes_{i=1}^{n}L\otimes \L^{p_i}W$ is the cone on the images of $f$ and $f_s$ in $\P \left(\bigotimes_{i=1}^{n} \L^{p_i}T^{\vee} \right)\cong \P \left(\bigotimes_{i=1}^{n} L\otimes \L^{p_i}T^{\vee} \right)$. This conical affine variety and the associated projective variety will both be called the \emph{multi-focal variety of $I$}, and they are to be distinguished by context. The points of these varieties are called \emph{multi-focal tensors}. The maps $f_w(I), f(I), f_s(I)$ will be called \emph{multi-focal maps}.
\end{remarkN}
\subsubsection{Symmetry reduction}\label{symmetryred}
\begin{definition}
The \emph{reduction to a subgroup}  $G'\subset G$ will mean the system
\begin{align*}
 H'&=H\cap G'&		&F'=G'/K_T'	\\
 K_L'&=K_L\cap G'&	&F_w'=G'/K_{LTW}'	\\
 K_T'&=K_T\cap G'&	&F_t'=G'/K_W'\\
 K_W'&=K_W\cap G'&	&F_s'=G'/\widetilde{K}_{T,W}\\
 \widetilde{K}_{T,W}'&=\widetilde{K}_{T,W}\cap G'&	&f_w'=f_w|_{F_w^{'n}}	\\
 K_{LTW}'&=K_{LTW}\cap G'&	&f'=f|_{F^{'n}}\\
 &		&	&f_s'=f_s|_{(F/k^{\times})^{'n}}
\end{align*}
\end{definition}
Note that $F'$, $F_w'$, $F_t'$, and $F_s'$ are also identified with the $G'$ orbits of the basepoints in $F,F_w,F_t,F_s$.
\begin{remark} (Terminology) In the presence of symmetry reduction, the terms \emph{multi-focal maps of $I$}, \emph{multi-focal variety of $I$}, and \emph{multi-focal tensors of $I$} are modified to mean the maps $f'_w, f', f'_s$, their image affine and projective varieties, and the points of these images.
\end{remark}
\begin{example}\label{affred} (Affine reduction) The reduction to the affine group entails a great deal of simplification.  We select some element $l\in V^{\vee}$ such that $l(v_0)\neq 0$, and define the affine group $\operatorname{Aff}$ to be the subgroup of $\GL(V)$ fixing $l$.  In this case we restrict our attention to $\mathbb{A}:=\operatorname{Aff}/H'$, which can be regarded as the affine chart for $\P V$ which is the complement of the projectivization of the hyperplane in $V$ equal to the kernel of $l$.
\end{example}
\begin{proposition} Under affine reduction $G'=\operatorname{Aff}$,
\begin{align*}
 & K_L' = H' (=\widetilde{K}_{L}') \\
 & H'\cong\{1\}\times \GL(W) \cong \GL(W)\\
 & K_T' = K_W' = K_{LTW}' = \{1\} \\
 & \widetilde{K}_{T,W}' = \{1\} \times k^{\times} \cong k^{\times}
\end{align*}
\end{proposition}
\begin{proof} Since $\operatorname{Aff}$ fixes $l$, and $l(v_0)\neq 0$, the scalars by which $H'=\operatorname{stab}([v_0])\subset \operatorname{Aff}$ act on $v_0$ must all equal to 1.  This proves the first equation.
The kernel of $l$ in $V$ is preserved by $\operatorname{Aff}$ and splits $0\rightarrow [v_0]\rightarrow V\rightarrow W\rightarrow 0$ as a short exact sequence of $H'$-modules.  Thus $H'$ is contained in a subgroup of $\GL(V)$ isomorphic to $\{1\}\times \GL(W)$.  Every element of this group also evidently fixes $l$, and so belongs to $H'$.  This proves the seconds equation.
It follows from $H'=K_L'$ that $K_T'=K_{LT}'$ and $K_{W}'=K_{LW}'$. By the 2-out-of-3 property in Proposition \ref{twooutofthreeP}.\ref{twooutofthree}, both of the latter equal to $K_{LTW}'$. By the description $H'=\GL(W)$, $K_{LTW}'$ is the trivial group.  This proves the third equations.
For the last equation: Each element of $H'\cong \GL(W)$ which acts by scalars on $W$ must of course belong to the center $k^{\times}\subset \GL(W)$.
\end{proof}
\begin{corollary} Under affine reduction $G'=\operatorname{Aff}$,
\begin{align*}
 & F' = F_w' = F_t' = \operatorname{Aff}\\
 & F_s' = \operatorname{Aff} / k^{\times}
\end{align*}
\end{corollary}
That is, our notion of affine frames, weighted affine frames, and twisted affine frames all agree.
\begin{remark} This makes the affine case more familiar than $\P V$ from the point of view of homogeneous geometry, in the sense that a symmetry $a\in A$ is uniquely determined by its action on a single frame.
\end{remark}
\begin{corollary} (Multi-focal maps under special symmetry reduction)\label{gpdescent}
Let $G'\subset \GL(V)$ be any subgroup which acts freely on the space of weighted frames $F'_w$. That is, such that $K'_{LTW}={1}$) and so the basepoint of $F'_w$ specifies an isomorphism $F'_w\cong G'$.
Under reduction to $G'$, the rational map\footnote{Note that the isomorphism $\bigotimes_{i=1}^{n}L\otimes \L^{p_i}W \cong \bigotimes_{i=1}^{n}\L^{p_i}W$ is non-canonical and depends on a choice of non-zero element in $L=[v_0]^{\vee}$, for example the dual of $v_0$.} $f'_w(I)$ described in Theorem \ref{mainconstruction},
\begin{align*}
 f'_w(I):G^{'n}\rightarrow	\bigotimes_{i=1}^{n}L\otimes \L^{p_i}W &\cong \bigotimes_{i=1}^{n}\L^{p_i}W \\
 	&\cong \bigotimes_{i=1}^{n}\L^{p_i}T^{\vee}
\end{align*}
descends to a regular function on $G^{'n-1}\cong G' \backslash G^{'n}$ which we call $f''(I)$,
\begin{align*}
 f''(I):G^{'n-1}\rightarrow \bigotimes_{i=1}^{n}\L^{p_i}W \cong \bigotimes_{i=1}^{n}\L^{p_i}T^{\vee},
\end{align*}
\end{corollary}
\begin{proof} $f'_w$ descends to a map on $G'\backslash G'^{n}$ by $G'$-invariance. To obtain $f''$, we must compose the result with an isomorphism $G^{'(n-1)}\cong G'\backslash G^{'n}$. We select the isomorphism specified by the section of the $G'$ quotient:
\begin{align*}
G^{'n-1}&\rightarrow G^{'n} \\
(b_1,\dots,b_{n-1})&\mapsto(b_{n-1}\dots b_3b_2b_1,\, b_{n-1}\dots b_3b_2,\, b_{n-1}\dots b_3,\,\dots,\, b_{n-1},\, \operatorname{id})
\end{align*}
\end{proof}
\begin{remark} (Terminology). In the presence of a symmetry reduction satisfying the hypotheses of Corollary \ref{gpdescent}, we will generally use the map $f''(I)$ to describe the multi-focal variety of $I$ rather than $f'_w(I)$. The images of $f''(I)$ and of $f'_w(I)$ are the same, but the domain of $f''(I)$ is simpler.
\end{remark}
\begin{example}\label{eucred} (Euclidean reduction) We define the Euclidean group to be the subgroup $E\subset \operatorname{Aff}$ of the affine group $\operatorname{Aff}\subset \GL(V)$ (see Example \ref{affred}) which acts by special orthogonal transformations of $\ker(l)\cong W = V/[v_0]$, for some chosen non-degenerate quadratic form on this $k$ vector space.

Most authors consider only the case $k=\mathbb{R}$ and stipulate that the quadratic form is positive-definite. We shall do the same, except by remarking here that much of what applies to these real Euclidean spaces applies more generally.

Note that the affine space $\mathbb{A}\subset \P V$ possesses a Riemannian metric, unique up to a uniform scale, for which the group $E$ acts by isometries; the usual flat metric.
\end{example}
% definition: multi-focal variety, projective, affine reduction, Euclidean reduction
% subsection on stability of points in Fn, Fwn ... derived from Dolgachev's worked examples
\section{Multi-focal tensors}
\subsubsection{Examples in low dimension}
As a warm-up, before undertaking the calculations in dimension $\dim V = 4$ involving the invariants $I_2,I_3,I_4$ which are the principal motivation of this paper, let's write the formulas for the multi-focal maps and varieties for invariants of skew-symmetric tensors in dimensions $\dim V = 2,3$.

\begin{remark} We shall generally consider only the invariants $I\in \L^{p_1+1}V^{\vee}\otimes .. \otimes \L^{p_n+1}V^{\vee}$ which can not be expressed as the product of invariants
\[ I_0 \in \L^{p_1+1}V^{\vee}\otimes .. \otimes \L^{p_i+1}V^{\vee},\qquad  I_1\in \L^{p_{(i+1)}+1}V^{\vee}\otimes \L^{p_{n}+1}V^{\vee} \]
This is because the multi-focal map $f_I$ would amount to the tensor product of the maps $f_{I_0}$ and $f_{I_1}$.
\end{remark}
\begin{example} ($\dim V = 2$)

Consider the exterior product $I\in\L^{1}k^{2\vee}\otimes \L^{1}k^{2\vee}$ (with respect to a choice of isomorphism $\L^{2}k^2 \cong k$). Use the bases introduced in section \ref{basisexpr}, and use the notation $'$ to indicate basis elements in the second factor of $\L^{1}k^{2\vee}$.\footnote{Hopefully this use of $'$ is not easily confused with the use of $'$ to indicate symmetry reduction as in $G',f',f''$.} Then
\[ I = e_0 \otimes e'_1 - e_1 \otimes e'_0\]
Apply the tensor product of the maps $\psi$ calculated in Example \ref{k2p0}:
\begin{align*}
\psi\otimes \psi' (I) &= \tau(I)\\
&=f(I)(g,g')\\
&= (g_{00}g'_{10}-g_{10}g'_{00})\otimes \tilde{e}_0^{2} \otimes 1 \otimes 1 
\end{align*}
(Here $g,g'$ can be considered elements of $GL(V)/H$.)

Although the 2-focal variety corresponding to this $I$ is evidently the entire space $k\cong \L^{0}(k^{2}/[v_0])^{\vee}\otimes\L^{0}(k^{2}/[v_0])^{\vee}$, we shall see that this example is actually not entirely trivial.

Consider the symmetry reduction to the 1-dimensional translation subgroup $G'$ of $\GL(2)$ consisting of elements of the form
\begin{align*}
& \begin{bmatrix} g_{00} & g_{01} \\ g_{10} & g_{11} \end{bmatrix}
= \begin{bmatrix} 1 & 0 \\ c & 1 \end{bmatrix}
\end{align*}
Let $g,g'\in G'$ be two such elements.
\[ \tau'(I) = f'(g,g')=(c' - c) \otimes \tilde{e}_0^2 \otimes 1 \otimes 1 \]
% That is, if $g$ denotes an element of the translation group representing an equivalence class of elements of $(G')^2$, as in Corollary \ref{gpdescent},
% \[ f''(I)(g) = g_{10} \otimes \tilde{e}_0^2 \otimes 1 \otimes 1 \]
% Thus the value of $f''(I)$ exactly determines the element of $G'$ representing the 2-tuple of $G'$-reduced frames in $\P^1$.
Thus the value of the symmetry-reduced multi-focal map $f'$ on the pair $(g,g')$ exactly encodes the relative position of the two symmetry-reduced frames of $\P^{1}$ represented by the elements $g,g'$.

Now consider the symmetry reduction to the 1-dimensional rotation subgroup $G'=SO(2)\subset \GL(2)$, consisting of elements of the form
\begin{align*}
& \begin{bmatrix} g_{00} & g_{01} \\ g_{10} & g_{11} \end{bmatrix}
= \begin{bmatrix} a & -b \\ b & a \end{bmatrix}
\end{align*}
such that $a^2+b^2 = 1$.

Let $g,g'\in SO(2)$ be two such elements.
\[ f'(I)(g,g')= (ab'-ba') \otimes \tilde{e}_0^2 \otimes 1 \otimes 1  \]
In the case $k=\mathbb{R}$, $ab'-ba'$ equals to $\sin \theta $, where $\theta$ is the angle between vectors
\[ \begin{bmatrix} a \\ b \end{bmatrix}\quad\begin{bmatrix} a' \\ b'\end{bmatrix} \]
Thus the relative position of the two frames represented by a pair of elements of $SO(2)$ is exactly encoded in the value that $f'(I)$ takes on this pair.
\end{example}
\begin{example} ($\dim V = 3$)
Consider the exterior product $I\in \L^{1}k^{3\vee}\otimes \L^{2}k^{3\vee}$ (with respect to a choice of isomorphism $\L^{3}k^{3}\cong k$). Use the bases introduced in section \ref{basisexpr}, and again use the notation $'$ to indicate the second factor, $\L^{2}k^{3\vee}$. Then
\[ I = e_0\otimes e'_1e'_2 - e_1\otimes e'_0e'_2 + e_2\otimes e'_0e'_1\]
For simplicity, we shall omit from the notation the factors of $\tilde{e}_0$. That is, the factor of $L^{n}$ from 
\[ \tau(I) \in \O_{(\GL(V))^{n}} \otimes L^{n}\otimes \bigotimes_{i=1}^{n} \L^{p_i}{W} \]
Apply the tensor product of the maps $\psi$ calculated in Examples \ref{k3p0} and \ref{k3p1}:
\begin{align*}
\psi\otimes\psi'(I) &=  \tau(I) \\
&= g_{00}(g'^{01}_{12}\otimes \tilde{e}_1+g'^{02}_{12}\otimes \tilde{e}_2) \\
& -g_{10}(g'^{01}_{02}\otimes \tilde{e}_1+g'^{02}_{02}\otimes \tilde{e}_2) \\
& +g_{20}(g'^{01}_{01}\otimes \tilde{e}_1+g'^{02}_{01}\otimes \tilde{e}_2)
\end{align*}
Once again the multi-focal variety is the entire space $k^{2}\cong \L^{0}(k^{3}/[v_0])\otimes \L^{1}(k^3/[v_0])$.

The reader may wish to reduce the symmetry group to the affine group, the hyperbolic plane isometry group, $SO(3)$, or the Heisenberg group, and then determine how much information about the relative position of a pair of symmetry-reduced frames is encoded in the resulting multi-focal tensor elements. I have not done so.

For encodings of the relative positions of 3-tuples of frames, one could repeat this procedure with the 3-fold exterior product
\begin{align*} 
 I &\in \L^{1}k^{3\vee}\otimes \L^{1}k^{3\vee} \otimes \L^{1}k^{3\vee} 
\end{align*}
However, in this case the ambient tensor space containing the multi-focal varieties is again 1-dimensional, so that these varieties are surely the whole space and they can contain only ``1-dimension's worth'' of information about the 3-frame configurations.
\end{example}
\subsection{Main Application}\label{mainapplication}\label{mainappsec}
\subsubsection{Obtaining the multi-focal tensors}
\begin{observation}  Assume that $\operatorname{dim}V=4$, and $k=\mathbb{R}$.
\begin{enumerate}
\itemsep0em
\item{\label{bif} Let $I_2\in \L^{2}V^{\vee}\otimes\L^{2}V^{\vee}$ be the exterior product, with respect to a choice of isomorphism $\L^{4}V^{\vee}\cong k$.

The values of the function $f_w(I_2)$ (or $f(I_2)$, or $f_w(I_2)$) are the \emph{fundamental matrices} or \emph{bifocal tensors} appearing in (\cite{luong}, \cite{faug93}, \cite{lf2}, \cite{fl}, \cite{hz}).

The values of the Euclidean reduction $f''(I_2):E\rightarrow T^{\vee}\otimes T^{\vee}$ with respect to the Euclidean isometry group $E$ are the \emph{essential matrices} appearing in (\cite{lh},\cite{dmz}).}
\item{\label{trif} The space of relative $\GL(V)$-invariants in $\L^{3}V^{\vee}\otimes\L^{2}V^{\vee}\otimes\L^{3}V^{\vee}$ is 1-dimensional. Let $I_3$ be a non-zero such invariant.

The values of the function $f_w(I_3)$ (or $f(I_3)$, or $f_w(I_3)$ are the \emph{trifocal tensors} appearing in (\cite{faug93}, \cite{hz}).

The values of the Euclidean reduction $f''(I_3):E^{2}\rightarrow \L^{2}T^{\vee}\otimes T^{\vee}\otimes\L^{2}T^{\vee}$ are the \emph{(Euclidean) trifocal tensors} appearing in (\cite{faug93} page 454, \cite{hz}).}
\item{Let $I_4\in(\L^{3}V^{\vee})^{\otimes 3}$ be the dual of the 4-fold exterior product, with respect to a choice of isomorphism $\L^{4}V^{\vee}\cong k$.

The values of the function $f_w(I_4)$ (or $f(I_4)$, or $f_w(I_4)$ are the \emph{quadrifocal tensors} appearing in (\cite{faug93}, \cite{hz}).}
\end{enumerate}
\end{observation}
A precise demonstration of the above identifications requires a complete translation of the language of the citations into the algebraic language appearing here. For example, the ``cameras" there can be regarded as projective linear projections $\P V\dashrightarrow \P ^{2}$ and then identified with frames up to scale. No real insight is needed besides checking that the translation is faithful, so we omit the details.

Instead, we focus on the Euclidean reductions and in this case show that the final formula for (\ref{bif}) agrees with that of \cite{dmz}, and the final formula (\ref{trif}) agrees with that of \cite{faug93}.
\subsubsection{Euclidean bifocal tensors}
Using the bases introduced in section \ref{basisexpr}, the formula for $I_2$ is
\begin{align*}
  I_2 =\quad	&e_0e_1\otimes e'_2 e'_3 + e_2e_3\otimes e'_0 e'_1 \\
 	-	&e_0e_2\otimes e'_1 e'_3 - e_1e_3\otimes e'_0 e'_2 \\
 	+	&e_0e_3\otimes e'_1 e'_2 + e_1e_2\otimes e'_0 e'_3,
\end{align*}
where the $'$ indicates the basis of the second factor of $\L^{2}V^{\vee}\otimes \L^{2}V^{\vee}$.

Using a basis for $V$ extending $v_0$ to a basis for $W\cong V/[v_o]$, the elements of the Euclidean group $E\subset \GL(V)$, defined in Example \ref{eucred}, have the form
\[ \begin{bmatrix}
 g_{00}   & g_{01} & g_{02} & g_{03}     \\
 g_{10} & g_{11} & g_{12} & g_{13}\\
 g_{20} & g_{21} & g_{22} & g_{23}\\
 g_{30} & g_{31} & g_{32} & g_{33}\\
\end{bmatrix}=\begin{bmatrix}
 1   & 0      & 0      & 0     \\
 u_1 & r_{11} & r_{12} & r_{13}\\
 u_2 & r_{21} & r_{22} & r_{23}\\
 u_3 & r_{31} & r_{32} & r_{33}\\
\end{bmatrix} \]
where the matrix $r=[r_{ij}]$ satisfies $r^{t}r=\operatorname{id}$ and $\operatorname{det}r=1$.  It follows that under Euclidean reduction, the maps $\psi$ and $\psi'$ are specified as in Example \ref{firstDim4}:
\begin{align*}
[\psi_{01}]= & \begin{bmatrix} r_{11} & r_{12} & r_{13} \end{bmatrix} \\
[\psi_{02}]= & \begin{bmatrix} r_{21} & r_{22} & r_{23} \end{bmatrix} \\
[\psi_{03}]= & \begin{bmatrix} r_{31} & r_{32} & r_{33} \end{bmatrix} \\
[\psi_{12}]= & \begin{bmatrix} r_{21}u_1-r_{11}u_2	&	r_{22}u_1-r_{12}u_2	&	r_{23}u_1-r_{13}u_2\end{bmatrix} \\
[\psi_{13}]= & \begin{bmatrix} r_{31}u_1-r_{11}u_3	&	r_{32}u_1-r_{12}u_3	&	r_{33}u_1-r_{13}u_3\end{bmatrix} \\
[\psi_{23}]= & \begin{bmatrix} r_{31}u_2-r_{21}u_3	&	r_{32}u_2-r_{22}u_3	&	r_{33}u_2-r_{23}u_3\end{bmatrix} \\
&\\
[\psi_{01}']= & \begin{bmatrix} r'_{11} & r'_{12} & r'_{13} \end{bmatrix} \\
[\psi_{02}']= & \begin{bmatrix} r'_{21} & r'_{22} & r'_{23} \end{bmatrix} \\
[\psi_{03}']= & \begin{bmatrix} r'_{31} & r'_{32} & r'_{33} \end{bmatrix} \\
[\psi_{12}']= & \begin{bmatrix} r'_{21}u'_1-r'_{11}u'_2	&	r'_{22}u'_1-r'_{12}u'_2	&	r'_{23}u'_1-r'_{13}u'_2\end{bmatrix} \\
[\psi_{13}']= & \begin{bmatrix} r'_{31}u'_1-r'_{11}u'_3	&	r'_{32}u'_1-r'_{12}u'_3	&	r'_{33}u'_1-r'_{13}u'_3\end{bmatrix} \\
[\psi_{23}']= & \begin{bmatrix} r'_{31}u'_2-r'_{21}u'_3	&	r'_{32}u'_2-r'_{22}u'_3	&	r'_{33}u'_2-r'_{23}u'_3\end{bmatrix} \\
\end{align*}
We apply $\psi\otimes \psi'$ to $I_2$ and then apply the section $(r,u)\mapsto ((r,u), (\operatorname{id},0))$, as in the definition of $f''$ appearing in Corollary \ref{gpdescent}. The map $f''(I):E\rightarrow T^{\vee}\otimes T^{\vee}$ is given by the formula (with respect to the basis for $T^{\vee}$ introduced in Example \ref{firstDim4}):
\begin{align*}
  (r,u)\mapsto	&	&	&\begin{bmatrix} r_{31}u_2-r_{21}u_3	&	r_{32}u_2-r_{22}u_3	&	r_{33}u_2-r_{23}u_3\end{bmatrix}\otimes \begin{bmatrix}1 & 0 & 0\end{bmatrix} \\
				&	& -	&\begin{bmatrix} r_{31}u_1-r_{11}u_3	&	r_{32}u_1-r_{12}u_3	&	r_{33}u_1-r_{13}u_3\end{bmatrix}\otimes \begin{bmatrix}0 & 1 & 0\end{bmatrix} \\
 				&	& +	&\begin{bmatrix} r_{21}u_1-r_{11}u_2	&	r_{22}u_1-r_{12}u_2	&	r_{23}u_1-r_{13}u_2\end{bmatrix}\otimes \begin{bmatrix}0 & 0 & 1\end{bmatrix} \\
 & & & \quad\\
 			=	&	&	&\begin{bmatrix} r_{31}u_2-r_{21}u_3	&	-(r_{31}u_1-r_{11}u_3)	& r_{21}u_1-r_{11}u_2 \\
 										r_{32}u_2-r_{22}u_3 &	-(r_{32}u_1-r_{12}u_3)	& r_{22}u_1-r_{12}u_2 \\
 										r_{33}u_2-r_{23}u_3 &	-(r_{33}u_1-r_{13}u_3)	& r_{23}u_1-r_{13}u_2
 						\end{bmatrix}\\
 & & & \quad\\
 			=	&	&	&r^{t}a
\end{align*}
where
\[ a = \begin{bmatrix}0	&	u_3	&	-u_2	\\	-u_3	&	0	&	u_1	\\	u_2	&	-u_1	&	0\end{bmatrix}\]
This is exactly the formula appearing in \cite{dmz}.
\subsubsection{Euclidean trifocal tensors}\label{explicitEucTri}
Adopting the abbreviations $e_{ij}=e_{i}e_{j}$ and $e_{ijk}=e_{i}e_{j}e_{k}$, the formula for $I_3$ is 
\[ I_3 =\begin{bmatrix} e_{123} & \matminus e_{023} & e_{013} & \matminus e_{012} \end{bmatrix}
\begin{bmatrix} 0 & e'_{01} & e'_{02} & e'_{03} \\ 
				\matminus e'_{01} & 0 & e'_{12} & e'_{13} \\
				\matminus e'_{02} & \matminus e'_{12} & 0 & e'_{23} \\
				\matminus e'_{03} & \matminus e'_{13} & \matminus e'_{23} & 0
\end{bmatrix}
\begin{bmatrix} e''_{123} \\ \matminus e''_{023} \\ e''_{013} \\ \matminus e''_{012}\end{bmatrix}\]
where $'$ and $''$ indicate the second and third factors of $ \L^{3}\otimes \L^{2} \otimes \L^{3}$, and the matrix entry multiplication is tensor product.

See \cite{grs} for methods of calculation of such invariants in general.

For convenience we will use a slightly different convention than the one stipulated in the proof of Corollary \ref{gpdescent} for the isomorphism $E\backslash E^{3}\cong E^{2}$ and the section of the projection $E^{3}\rightarrow E\backslash E^{3}\cong E^{2}$.  Namely,
\begin{align*}
E\backslash E^{3} &\rightarrow E^{2}\\
(g_1,g_2,g_3) &\mapsto (g_1^{-1}g_2,g_3^{-1}g_2)\\
&\\
E^{2} &\rightarrow E^{3}\\
(g_1,g_2) &\mapsto (g_1^{-1},\operatorname{id},g_2^{-1})
\end{align*}
The calculation of the restriction of $\psi\otimes\psi'\otimes \psi''(I_3)$ to this section is a lengthy but straightforward ring computation. The intermediate steps are omitted. The most convincing calculation is computer-assisted; I used Macaulay2. The result, however, is very simple:
\begin{align}\label{ruswform}
\begin{split}
	& f''(I_3): E^{2}\rightarrow \L^{2}T^{\vee}\otimes T^{\vee}\otimes\L^{2}T^{\vee}\\
	& ((r,u),(s,w))\mapsto -r\otimes w + u \otimes s
\end{split}
\end{align}
In this formula $u$ and $w$ are regarded as elements of $T$, and we have made the identifications\footnote{Note that these identifications are isomorphisms of $SL(T)$ representations, and in particular of $H=H_E=SO(3)$ representations.}:
\begin{align*}
	& \L^{2}T^{\vee}\otimes T^{\vee} \cong T\otimes T^{\vee} \cong \operatorname{End}T\\
	& T^{\vee}\otimes \L^{2}T^{\vee} \cong T^{\vee}\otimes T \cong \operatorname{End}T
\end{align*}
in order to express $r$ as an element of $\L^{2}T^{\vee}\otimes T^{\vee}$ and $s$ as an element of $T^{\vee}\otimes \L^{2}T^{\vee}$.

To obtain the formula (\cite{faug93} page 454, item 8.48), switch to index notation. For the reader's convenience we recall without explanation the formula of \cite{faug93} in the original notation:
\[ \mathbf{G}^{n}_i = \mathbf{t}_{ij}\mathbf{R}^{(n)T}_{ik} +\mathbf{R}^{(n)}_{ij}\mathbf{t}^{T}_{ik}  \]
\begin{remark} Here is one difficulty one may encounter in trying to write a non-computational proof of the formula \ref{ruswform}. It is not obviously linear in $(r,u)$ and in $(s,w)$; indeed, the formula which one would obtain by selecting the more naive section $(g_1,g_2) \mapsto (g_1,\operatorname{id},g_2)$, rather than the one we have chosen, is \emph{mixed quadratic-linear}, because it differs by an application of the group inversion map $E\rightarrow E$ applied separately to each component of $E^2$:
\begin{align*}
& (r,u)\mapsto (r^{t},-r^{t}u)\\
& (s,w)\mapsto (s^{t},-s^{t}w)
\end{align*}
\end{remark}
\subsection{Geometric consequences}\label{geomconseqsec}
\subsubsection{Geometric properties of invariants $I_2,I_3,I_4$}
Recall that $\dim V=4$. The following geometric properties are well-known:
\begin{proposition} \label{geometricinterpretation}\emph{(Geometric properties of $I_2,I_3,I_4$)}
\begin{enumerate}
\itemsep0em
\item{\label{ii2}Suppose that $a,b\in \L^{2}V$ are decomposable, so that they represent 2 lines in the 3-dimensional projective space $\P V$.

$I_2(a\otimes b)=0$ if and only if the 2 lines meet.}
\item{\label{ii3}Suppose that $p,q\in \L^{3}V$ are non-zero. They are automatically decomposable, so that they represent 2 planes in $\P V$. Let $a \in \L^{2}V$ be decomposable so that it represents a line.

$I_3(p\otimes a \otimes q)=0$ if and only if the intersection of the 2 planes meets the line (Equivalently, the planes and line have at least 1 point in common).}
\item{\label{ii4}Suppose that $p,q,r,s\in \L^{3}V$ are non-zero. They are automatically decomposable, so that they represent 4 planes in $\P V$.

$I_4(p\otimes q\otimes r\otimes s)=0$ if and only if the 4 planes have at least 1 point in common.}
\end{enumerate}
\end{proposition}
\subsubsection{Linear sections of the multi-focal varieties}
% Consider the symmetry reduction from $\GL(V)$ to some affine subgroup $G'\subset A$, as described in section \ref{symmetryred}.
In this section we introduce a construction which, although it is motivated directly from the visual-geometric applications in the case $\dim V = 4$, applies in the generality of arbitrary invariants $I\in \L^{p_1+1}V^{\vee}\otimes .. \otimes \L^{p_n+1}V^{\vee}$, for arbitrary $\dim V$.

Let $b_1,..,b_n\in F$ be an $n$-tuple of frames for tangent spaces of $\P V$ at disjoint basepoints. 

Let $t$ denote any lift of $f(I)(b)$ to an element of $\bigotimes_{i=1}^{n}\L^{p_i}T^{\vee}$. See Theorem \ref{mainconstruction} for the definition of $f$.\footnote{Note again that we have made a non-canonical choice of isomorphism $\bigotimes_{i=1}^{n}L\otimes \L^{p_i}W \cong \bigotimes_{i=1}^{n}\L^{p_i}W$ depending on a choice of non-zero element in $L$.}

% If desired one can set $t=t(I)([b])$ using the map $t(I)$ defined in the construction of Corollary \ref{gpdescent} in the presence of affine symmetry reduction.

Let $(c_1,c_2,..,c_n)\in \L^{p_1}T \times\L^{p_2}T\times..\times \L^{p_n}T$ be decomposable elements, so that they represent a $(p_1-1)$-plane, a $(p_2-1)$-plane, etc. in $\P T$.

Let $(d_1,d_2,..,d_n)\in \L^{p_1+1}V\times \L^{p_2+1}V \times .. \times \L^{p_n+1}V$ denote decomposable elements representing the $p_1$-plane, $p_2$-plane ... and $p_n$-plane in $\P V$ to which the planes $c$ are transported by the frames $b$.\footnote{Note that this uses the choice of isomorphism $i:k^{n}\rightarrow T$ fixed in Remark \ref{framesisochoice}.}

\begin{proposition}
$I(d_1\otimes .. \otimes d_n)=0$ if and only if $t(c_1\otimes .. \otimes c_n)=0$.
\end{proposition}

\begin{proof} Up to scale, $t$ is by definition the pullback of $I$ under the assignment $c_1\otimes .. \otimes c_n\mapsto d_1\otimes .. \otimes d_n$. The result follows.

\end{proof}

Equivalently: Consider $c_1\otimes .. \otimes c_n$ as a linear functional on the vector space containing $t$. Then $t$ belongs to the hyperplane which is the kernel of this functional if and only if $I$ vanishes on all of the tuples of planes $d$ which are the lifts, determined by the frames $b$, of the tuple $c$.

\begin{corollary} \label{linearsections}Consider the case $\dim V =4$ as in section \ref{mainapplication}.

\begin{enumerate} 
\itemsep0em
\item{\label{essentialmatrixconstraintlinear}Set $I=I_2$. Fix a 2-tuple of frames $b_1,b_2\in F$ at different basepoints of $\P V$. Suppose that representatives $c_1,c_2\in T$ of points in $\P T$ correspond via the frames $b$ to a pair of lines in $\P V$ which are known to intersect.

Then $t=t(b)$ belongs to the kernel of $c_1\otimes c_2$. The converse also holds.}
\item{Set $I=I_3$. Fix a 3-tuple of frames $b_1,b_2,b_3 \in F$ at different basepoints of $\P V$. Suppose that representatives $c_1,c_3\in \L^{2}T$ of 2 lines in $\P T$ and a representative $c_2\in T$ of a point in $\P T$ correspond via the frames $b$ to a pair of planes in $\P V$ and a line of $\P V$ which are known to have a point of mutual intersection.

Then $t=t(b)$ belongs to the kernel of $c_1\otimes c_2 \otimes c_3$. The converse also holds.}
\item{Set $I=I_4$. Fix a 4-tuple of frames $b_1,b_2,b_3,b_4\in F$ at different basepoints of $\P V$. Suppose that representatives $c_1,c_2,c_3,c_4\in \L^{2}T$ of 4 lines in $\P T$ correspond via the frames $b$ to 4 planes in $\P V$ which are known to have a point of mutual intersection.

Then $t=t(b)$ belongs to the kernel of $c_1\otimes c_2\otimes c_3\otimes c_4$. The converse also holds.}
\end{enumerate}
\end{corollary}
\subsection{Practical usage}\label{practicalsec}
\subsubsection{Determination of a finite frame configuration}
The multi-focal varieties derived from the invariants $I=I_2,I_3,I_4$ are used in practice as follows.

Consider a fixed but unknown configuration $b$ of $n=2,3,\text{or } 4$ frames --``cameras". Let $t$ be any lift of the elements $f(I)(b)$ or $f_s(I)(b)$ specified in Theorem \ref{mainconstruction} to the tensor space $\L^{p_1}T^{\vee}\otimes .. \otimes \L^{p_n}T^{\vee}$ (or of the element $f''(I)(b)$, if desired, in the presence of symmetry reduction). $t$ is well-defined up to scale.

A collection $c$ of 2, 3, or 4 points or lines of $\P T$ known to be $b$-related to a configuration $d$ of lines and planes, in the ambient 3-dimensional projective space, of the types enumerated in Proposition \ref{geometricinterpretation}, determine linear constraints on the unknown tensor element $t$ as described in that Proposition. Typical such configurations $c$ arise from configurations $d$ which are loci of ``visual rays" connecting the basepoints of the 2,3, or 4 frames $b$, the ``observers'', to fixed points or lines in the ambient projective 3-space. In the terms of visual geometry, the configurations $c$ are directly measurable as the visual images of points or lines in space, belonging to 2-dimensional visual planes, obtained from $n=2,3,\text{or }4$ points of view.

Sufficiently many such linear constraints determine $t$ uniquely up to scale. Then some attempt is made to determine the frame configuration $b$ from the value of $t(b)$ so obtained.
\begin{example} Under Euclidean reduction in the case $I=I_2$, a linear constraint on $t(b)$ is specified by the knowledge of two points of $\P T$ satisfying the hypothesis in Corollary \ref{linearsections}(\ref{essentialmatrixconstraintlinear}). It turns out that there is a closed formula for the pre-images $b$ of $t(b)$ in terms of $t(b)$. It can be ascertained from (\cite{faug93} page 284).
\end{example}
This technique is normally extended to a configuration of $n$ frames with $n>4$, by expressing it as the union of sufficiently many pairs, triples, and quadruples for which the ``visual observations" $c$ are available.
\begin{remarkN}In a completely different direction, one might hope to find $n$-factor invariants $I$ for very large $n$, for which a geometric interpretation can be found analogous to the interpretations of the invariants $I_2,I_3,I_4$ explained in Proposition \ref{geometricinterpretation}. Then the exact same procedure, which used $I_2,I_3,I_4$ to find $2,3,4$-frame configurations, could be attempted with $I$ instead. A notable candidate is the ``line-complex" invariant of $(\L^{2}V^{\vee})^{\otimes 6}$, not derived in any way from the usual exterior product or meet, whose vanishing indicates that 6 lines in $\P^{3}$ belong to a linear line complex (see \cite{jessop}). The formula for this invariant is the $6\times 6$ matrix determinant.
\end{remarkN}
\subsubsection{Determination of a moving frame}
In practice finite frame configurations often arise as approximations to a smooth path $\gamma:[0,1]\rightarrow F$, covering a smooth immersed path of basepoints $p:[0,1]\rightarrow \P V$. In this case, fix an $n$-factor invariant of skew-symmetric tensors $I$, and consider the product immersed manifolds
\begin{align*}
& M:=\left(p([0,1])\right) ^{n}\subset \left(\P V\right)^{n}\\
& N:=\left(\gamma([0,1])\right)^{n}\subset F^{n}
\end{align*}
Assume that the projection $N\rightarrow M$ is a diffeomorphism.

Suppose that the restriction of the multi-focal map $f(I)$ from $F^{n}$ to $N$ is given with respect to the natural coordinates $[0,1]^{n}$ on $M\cong N$. By construction, up to pointwise scale this is the expression $\eta_N$ of the differential form $\eta(I)$ on $(\P V)^{n}$ along $M$ with respect to the framing $N$ of the tangents spaces of $(\P V)^{n}$:
\[ \eta_N:[0,1]^{n}\rightarrow \bigotimes_{i=1}^{n}\L^{p_i}T^{\vee} \]

The problem is to find $(M,N,p,\gamma)$ from $\eta_N$.

Note the similarity to the problem of non-abelian integration:
\begin{theorem} \emph{(e.g. \cite{sharpe}) (Non-abelian fundamental theorem of calculus)}

Let $G$ be a simply-connected, finite-dimensional real Lie group with Lie algebra $\mathfrak{g}$.

Let $M$ be a smooth manifold, and $\omega_M$ a differential 1-form on $M$ with values in a Lie algebra $\mathfrak{g}$.

Let $\omega_G$ denote the left (or right)-invariant Maurer-Cartan form of $G$.

Then there is a smooth map $p:M\rightarrow G$ such that $p^{*}\omega_G = \omega_M$, unique up to left (or right) translation in $G$, if and only if $d\omega_M+\tfrac{1}{2}[\omega_M,\omega_M]=0$.
\end{theorem}

(Note: This theorem is probably originally due to Elie Cartan, or an even earlier author).

The intuitive idea is that the values of $\omega_M$ specify how the tangent spaces of $M$ should be situated in $G$ with respect to the tautological framing of the tangent bundle of $G$.

An analogous integrability or consistency criterion for the multi-focal element field $\eta_N$ on $[0,1]^{n}$, guaranteeing the existence of the product manifolds $M$ and $N$ and the maps $p$ and $\gamma$, is not known.

\begin{remark} Despite appearances, there is one strong dissimilarity with the situation of non-abelian integration. The pullback $p^{*}\omega_G$ involves the first-order derivatives of $p$. It depends on the tangent spaces of $p(M)$. On the other hand, the data $\eta_N$ on $[0,1]^{n}$, the ``pullback" of $\eta$, is of order zero in  $\gamma$ and $p$; it depends only on their values, not their derivatives.

This is because the frame field $\gamma$ plays the role in our setting that is played in the setting of non-abelian integration by the tangent map of the map $p$, also known as the Jacobian $J(p)$ or differential $dp$. On the other hand, the situations \emph{really are} analogous in case the frame field is assumed to be differentially related to the path $p$. For example, in the case of Euclidean reduction, if $\gamma$ is assumed to be the Frenet frame of the path $p$.
\end{remark}
In the absence of such a criterion, current numerical algorithms for determining $p$ and $\gamma$ from $\eta_N$, e.g. the Theia Vision Library \cite{theia} or the Open Multiple View Geometry library \cite{openmvg}, use a feedback-based error-minimization technique called ``bundle adjustment". In practice this strategy is known to be far from optimal. Indeed, it employs the same steps when the configurations of frames are discrete that it does when the configuration of frames is an approximation to a smooth path of frames. So it does not even make use of the assumption of continuity of the paths (although, see \cite{astromheyden}).

Such algorithms would certainly be improved by a step which enforces consistency of $\eta_N$ before attempting to reconstruct the maps $p$ and $\gamma$.
\subsection{Constraints}\label{constraintssec}
It is important to know equations describing the multi-focal varieties, for some fixed invariant $I$, for at least two reasons:
\begin{enumerate}
\itemsep0em
\item{To determine a given unknown multi-focal tensor $t$, in principle fewer linear constraints are needed if the constraints are augmented by the (always non-linear) constraints satisfied by the entire multi-focal variety.}
\item{A procedure for determining $b$ from its multi-focal tensor $t$ is certainly more likely to succeed if the value of $t$ estimated by solving a linear system is known to lie on the multi-focal variety.}
\end{enumerate}
\begin{remarkN} A system of polynomial equations describing a given multi-focal variety would be useful for testing whether or not an estimated value of $t$ is actually of the form $t(b)$ for some $b$. By employing a certain amount of error analysis, one might even hope to use the formulas appearing in the equations to estimate how far a given value of $t$ is from lying on the multi-focal variety.

\emph{However}, it would be even more useful to have a \emph{means} of enforcing the constraints; a polynomial mapping from the ambient tensor space to itself which is a retraction onto the multi-focal variety. In algebraic terms, this would amount to a splitting of the short exact sequence of rings
\[0\rightarrow J\rightarrow R\rightarrow R/J\rightarrow 0\]
where $J$ denotes the radical ideal of functions vanishing on the variety, and $R$ denotes the coordinate ring of the ambient tensor space. I do not know of any such algebraic retraction for any invariant $I$, or a proof that one does not exist.
\end{remarkN}
\begin{remarkN} Amazingly, the Theorems 1 and 2 of \cite{semialgretract} seem to imply that a \emph{semi-algebric} retraction does exist.
\end{remarkN}
\begin{remarkN} Even more useful than a retraction $Z\rightarrow Z$ of the tensor space $Z$ onto the multifocal variety $f_I(F^{n})\subset Z$ would be a direct algebraic or semi-algebraic map $Z\rightarrow F^{n}$ which is a right inverse of $f_I$. This is a rather serious matter in practice. For example, just because one knows that an element of $Z$ belongs to $f''_{I_{3}}(E^{2})$, it does not follow by any means that one knows an element of $E^{2}$ mapping to it.
\end{remarkN}
\subsubsection{Constraints on the general and Euclidean bifocal tensors}
No doubt inspired by \cite{lh}, Demazure proved in \cite{dmz} that the conical complex affine variety $\mathcal{V}$ equal to the Zariski closure of the complexification of the set of real $3\times 3$ matrices of the form $m=r^{t}a$, where $r$ is real orthogonal and $a$ is real anti-symmetric, is irreducible with projectivization of degree $10$ in $\mathbb{CP}^{8}$. He also proved that the ideal $J$ generated by the cubic polynomial $\det m$ and the $3\cdot3=9$ cubic polynomials given by the entries of the matrix
\begin{align*}
c(m):=&\tfrac{1}{2}\operatorname{tr}(mm^{t})m  - mm^{t}m\\
=&\left(\tfrac{1}{2}\operatorname{tr}(mm^{t})\operatorname{Id}-mm^{t}\right)m  \\
=&m\left(\tfrac{1}{2}\operatorname{tr}(mm^{t})\operatorname{Id}-m^{t}m\right)
\end{align*}
has zero locus equal to this irreducible variety, though he does not seem to prove that this ideal is radical.

We mention the following fact, relevant to the proofs, as a way to introduce a certain important quartic polynomial $q(m)\in J$. The ideal generated by the polynomials $c$ does not contain $\det m$, but $\det m$ does belong to the radical of the ideal generated by $c$ and the polynomial
\begin{align*}
q(m):= \tfrac{1}{2}(\operatorname{tr}mm^{t})^2 -\operatorname{tr}\left[(mm^{t})^2\right]
\end{align*}
This is because of the equation
\begin{align}\label{frob}
\operatorname{tr}(cc^t)=-\tfrac{1}{2}\operatorname{tr}(mm^t)q+3\operatorname{det}m^3
\end{align}
Without indicating a textual reference, Demazure credits O. Faugeras with the crucial fact that a real matrix $m$ has a real factorization of the form $r^{t}m$ if and only if $\det(m)=0$ and $q(m)=0$. He then shows that the mapping $(r,a)\mapsto r^{t}a$ defines a degree 2 dominant map of real varieties $O(3)\times \mathbb{R}^3\rightarrow \mathcal{V}_{\mathbb{R}}$.

Since $\mathcal{V}$ has (complex) dimension 7 in $\mathbb{C}^9$ and the (real) dimension of $O(3)\times \mathbb{R}^3$ is 6, this is somewhat surprising. Evidently as a real submanifold $\mathcal{V}$ is not tranverse to the real plane $\mathbb{R}^9\subset \mathbb{C}^9$. The fact that the two polynomials $\det(m)$ and $q(m)$ cut out a real variety of codimension larger than 2 can be explained by the observation that the left-hand side of equation (\ref{frob}) is the Frobenius norm of $c$, the sum of the squares of the entries of $c$. Over the real numbers, the vanishing of this norm implies the vanishing of all of the entries.

This completes our discussion of constraints on the Euclidean 2-focal variety.

The general 2-focal variety, without Euclidean reduction, turns out to be just the determinant locus $\det m = 0$ (see e.g. \cite{fl}).
\subsubsection{Constraints on the general trifocal tensors} \label{constraintsEucTri}
Aholt and Oeding showed in \cite{ao} that the radical prime ideal of the Zariski closure of the complexification of the trifocal variety is generated by 10 polynomials of degree 3, 81 polynomials of degree 5, and 1980 polynomials of degree 6, by explicitly listing generating $\SL(T_{\mathbb{C}})^{\times 3}$-modules in the polynomial algebra. Consult \cite{alzatitortora} for a review of the literature concerning polynomial equations characterizing the general trifocal variety as a set. Here we review only selected aspects of this topic, and mention that a complete set of polynomial equations describing the \emph{Euclidean} trifocal variety as a set (the case of practical interest!) does not seem to be known.
\vspace{1pc}

Let us use the following notation. We consider a general element $t$ of the trifocal variety, the image of the map $f_w(I):F_w^{3}\rightarrow \L^{2\vee}T\otimes T^{\vee} \otimes \L^{2\vee}T$. Fix an arbitrary basis for $T$, with the corresponding bases for $T^{\vee}$ and $\L^{2}T^{\vee}$. By contracting $t$ with the three basis elements into the middle factor $T^{\vee}$, we obtain three tensor elements $t_1,t_2,t_3$. We regard them as $3\times 3$ matrices with respect to the basis for $\L^{2}T^{\vee}$.

For $x=(x_1,x_2,x_3)\in k^3$, we set $t(x) = x_1t_1+x_2t_2+x_3t_3$.
\begin{theorem} \emph{(Papadopoulo and Faugeras \cite{pf})} The elements $t$ of the trifocal variety (in the real case $k=\mathbb{R}$) satisfy the following two conditions:
\begin{enumerate}
\itemsep0em
\item{\label{determinatnal}$\operatorname{det}(t(x))=0$ for all $x=(x_1,x_2,x_3)\in \mathbb{R}^3$}
\item{\label{epi}The system of right kernels of the matrices $t(x)$ is 2-dimensional or less. The system of left kernels is also 2-dimensional or less.}
\end{enumerate}
\end{theorem}% In the statement of Theorem 1 of \cite{pf}, the authors claim that these conditions also characterize the ``manifold" of trifocal tensors as a set. The version of the trifocal variety which is an affine cone belonging to a tensor space is not a manifold, at least because it contains the cone point at the origin. Its projectivization may or may not be a manifold; no proof is given either way. The reader should consult their article to decide what exactly is proved there.
Note that the condition (\ref{determinatnal}) above is a system of 10 equations of degree 3 on the components of $t$. To obtain explicit formulas for them, expand the determinant.

The authors of \cite{pf} call the conditions (\ref{epi}) the \emph{epipolar constraints}. They do not supply equations for the coefficients of $t$ which are equivalent to these conditions. Under the further assumption that all $t(x)$ have rank 2 (and not rank 1 or 0), equivalent equations are given below.

\begin{proposition} \label{adjugate}Let $t_1,t_2,t_3$ be a basis for a 3-dimensional linear system of 3 by 3 matrices over a field $k$. Assume that all $t(x)=x_1t_1+x_2t_2+x_3t_3$, for non-zero $x\in k^3$, have rank exactly 2.

Then the system has common right kernel of dimension 2 or less if and only if the following system of 27 equations of degree 6 holds:
\begin{align*}
\sum_{\sigma\in S_3} \text{sign}(\sigma)a^{1}_{i\sigma(1)}a^{2}_{j\sigma(2)}a^{3}_{k\sigma(3)} = 0
\end{align*}
\noindent where $a^{1}$, $a^{2}$, and $a^{3}$ denote the classical adjoints \footnote{also called the adjugate, meaning the transpose of the matrix of cofactors} of the matrices $t_1$, $t_2$, $t_3$.
\end{proposition}
\begin{proof} Observe that in dimension 3, identifying $k^3\cong \L^{2}k^{3}$, the following formula holds:
\begin{align*}
\text{the classical adjoint}(x\otimes y + z\otimes w) = (y\wedge w)\otimes (x\wedge z)
\end{align*}
I do not know a proof of this formula besides direct calculation of all matrix entries, so it is omitted.

Evidently the left-hand factor $y\wedge w$, regarded as a linear functional on $k^3$ with values in $\L^{3}k^{3}$, describes the right kernel of $x\otimes y + z\otimes w$. The rank condition implies that there are $x_i,y_i,z_i,w_i\in k^3$ such that
\begin{align*}
t_1 &= x_1 \otimes y_1 + z_1 \otimes w_1\\
t_2 &= x_2 \otimes y_2 + z_2 \otimes w_2\\
t_3 &= x_3 \otimes y_3 + z_3 \otimes w_3
\end{align*}
Then
\begin{align*}
a^1 &= (y_1\wedge w_1) \otimes (x_1\wedge z_1)\\
a^2 &= (y_2\wedge w_2) \otimes (x_2\wedge z_2)\\
a^3 &= (y_3\wedge w_3) \otimes (x_3\wedge z_3)
\end{align*}
The common right kernel is described by the span of the $y_1\wedge w_1,y_2\wedge w_2,y_3\wedge w_3$. 

The (outer) exterior product of the left-hand factors of $a^1\otimes a^2\otimes a^3$ is
\[ \det(y_1\wedge w_1,y_2\wedge w_2,y_3\wedge w_3)\cdot (x_1\wedge z_1)\otimes(x_2\wedge z_2)\otimes(x_3\wedge z_3) \]
Since all $t_i$ have rank 2, no $x_i\wedge z_i$ is zero, so that the right-hand triple tensor product does not vanish identically. Therefore this expression vanishes identically if and only if the left-hand determinant does. This happens if and only if the common right kernel of the $t_1,t_2,t_3$ has dimension 2 or less.

On the other hand this expression is zero if and only if for all $i,j,k$:
\begin{align*}
\sum_{\sigma\in S_3} \text{sign}(\sigma)a^{1}_{\sigma(1)i}a^{2}_{\sigma(2)j}a^{3}_{\sigma(3)k} = 0
\end{align*}
This is because of the polarization formula for the determinant (see the preliminary chapter of \cite{dolg}).
\end{proof}
\begin{remark} The same proof proves the analogous fact that the common left kernel of the $t_1,t_2,t_3$ is sub-maximal rank if and only if
\begin{align*}
\sum_{\sigma\in S_3} \text{sign}(\sigma)a^{1}_{i\sigma(1)}a^{2}_{j\sigma(2)}a^{3}_{k\sigma(3)} = 0
\end{align*}
\end{remark}
\subsubsection{Constraints on the Euclidean trifocal tensor}\label{newconstraints}
To simplify the notation in what follows we make the identifications of $H'\cong SO(3)$-modules
\[\mathbb{R}^{3}\cong \mathbb{R}^{3*}\cong T\cong T^{\vee}\cong \L^{2}T^{\vee}\]
We use the description of the Euclidean trifocal tensors obtained in section (\ref{explicitEucTri}),
\[t = -r\otimes w + u \otimes s \in (\mathbb{R}^{3})^{\otimes 3}\]
where $u,w\in \mathbb{R}^3$ and $r,s\in SO(3)\subset \operatorname{Mat}_{3\times 3}(\mathbb{R})\cong\mathbb{R}^{3}\otimes \mathbb{R}^{3}$.

Define matrices $t_1,t_2,t_3$ as in the beginning of section \ref{constraintsEucTri}. Let $a_1,a_2,a_3$ denote their classical adjoints, or adjugates, and let $r_1,r_2,r_3$ and $s_1,s_2,s_3$ denote column vectors of matrices $r,s$. That is,
\begin{align*}
t_1 &= -r_1\otimes w + u\otimes s_1\\
t_2 &= -r_2\otimes w + u\otimes s_2\\
t_3 &= -r_3\otimes w + u\otimes s_3
\end{align*}
Since the matrix $r$ is special orthogonal, the usual cross-product (or wedge product) of $r_i$ and $r_j$, $i\neq j$, is equal to $\pm r_k$, where $k$ is the index from the set $1,2,3$ which is not $i$ or $j$, signed so that the list $(r_i,r_j,\pm r_k)$ comprises a basis with the same orientation as $(r_1,r_2,r_3)$. We adopt the notation $\pm r_k = r_{\{i,j\}}$ in this case (and similarly, $s_{\{i,j\}}$).
\begin{proposition} \emph{(Some algebraic properties of the Euclidean trifocal variety)}
For distinct $i,j,k$:
\begin{enumerate}
\itemsep0em
\item{\label{f1}$a_i =(s_i\wedge w)\otimes(r_i\wedge u)$}
\item{\label{f2}$t_i a_i     =(\operatorname{det}t_i)\operatorname{Id}=0$}
\item{\label{f3}$t_i a_j     =u\otimes(r_j\wedge u)\langle s_{\{i, j\}}, w\rangle$}
\item{\label{f4}$a_j t_i     =-(s_j\wedge w)\otimes w \langle r_{\{i, j\}}, u\rangle$}
\item{\label{f5}$t_i a_j t_i =(u\otimes w)\langle s_{\{i,j\}},w\rangle\langle r_{\{j,i\}}, u\rangle$}
\item{\label{f7}$t_i a_j t_k = (u\otimes w)\langle s_{\{i,j\}},w\rangle\langle r_{\{j,k\}},u\rangle$}
\item{\label{f8}$a_i t_j a_k = 0$}
\end{enumerate}
\end{proposition}
\begin{proof} (\ref{f1}) This follows from the first formula appearing in the proof of Proposition \ref{adjugate} (note the sign).

(\ref{f2}) The first equation here is one definition of the determinant. The vanishing of $\det(t_i)$ is implied by the fact that the $t_i$ are rank 2.

(\ref{f3}) Use the formula (\ref{f1}) for $a_j$ and the definition of $t_i$. The matrix product has two terms, one of which vanishes by $\langle w,s_j\wedge w\rangle=0$. For the coefficient of the remaining term use the triple product formula
\[ \langle s_i,s_j\wedge w\rangle = \langle w,s_i\wedge s_j\rangle =\langle w,s_{\{i,j\}}\rangle\]
(\ref{f4}) Use the same proof as the above, with the matrix product in the other order.

(\ref{f5}) Use one of the formulas, (\ref{f3}) for $t_ia_j$, or (\ref{f4}) for $a_jt_i$, multiplied on the right or left by the defining formula for $t_i$.

(\ref{f7}) Multiply the formula for $t_ia_j$ on the right by the formula for $t_k$.

(\ref{f8}) Multiply the formula for $a_it_j$ on the right by the formula for $a_k$. All terms vanish.
\end{proof}
\begin{corollary} (Braid-type relation)
For distinct $i,j$,
\[ t_i a_j t_i = t_j a_i t_j\]
\end{corollary}
\begin{proof} According to the previous proposition, these two terms are both equal to
\[(u\otimes w)\langle r_{\{i,j\}}, v\rangle\langle s_{\{j,i\}},w\rangle =(-1)(-1)(u\otimes w)\langle r_{\{j,i\}}, v\rangle\langle s_{\{i,j\}},w\rangle\]
\end{proof}
Now adopt the bar notation $\bar{ }$ to denote inversion in the Euclidean group, so that
\begin{align*}
& (\bar{r},\bar{u}) = (r^{t},-r^{t}u)\\
& (\bar{s},\bar{w}) = (s^{t},-s^{t}w)
\end{align*}
We omit the proof of the following formulas, and refer the reader instead to a computer algebra system like Macaulay2.
\begin{proposition} \hspace{1mm}
\begin{enumerate}
\itemsep0em
\item{$\text{classical adjoint}(a_1+a_2+a_3)= (u\otimes w)\langle \bar{u},\bar{w}\rangle$}
\item{The tensor belonging to $(\mathbb{R}^{3})^{\otimes 4}$ given by
\begin{align*}
\begin{pmatrix}
 -t_3a_2t_3 & t_2a_3t_1 & t_3a_2t_1 \\
 t_1a_3t_2 & -t_3a_1t_3 & t_3a_1t_2 \\
 t_1a_2t_3 & t_2a_1t_3 & -t_2a_1t_2 
\end{pmatrix}
\end{align*}
is equal to $u\otimes w \otimes \bar{w}\otimes \bar{u}$.

Here the first two factors, $u$ and $w$, correspond respectively to the row and column of the entries of this matrix, while the second two factors $\bar{w}$ and $\bar{u}$ correspond respectively to the row and column of the matrix itself.}
\end{enumerate}
\end{proposition}

% (\ref{f6}) 
% \item{\label{f6}$\operatorname{adj}(a_1+a_2+a_3) =(v\otimes w)\langle v,w\rangle$}
% \subsection{The inverse problem}
%When do we get an embedding?
%When is there an inverse map, at least rational?
% \subsubsection{Bifocal tensor}
% \subsubsection{Trifocal tensor}
%\subsubsubsection{geometric interpretation}		...geometric interpretation of multi-focal tensor elements
% \section{Focal surfaces of line congruences}
% \subsection{..}
% Monge; Lie?
% Wylcynski's students' theses
% Apparent contours of algebraic hypersurfaces: ref: Dolgachev's Classical Algebraic Geometry
% Arnold's pictures 
% Remark on Cipolla Giblin
% Definition of simple singularity set of the line complex of a surface
% Definition of the focal surface of a line congruence
% Claim that these operations are inverse
% Claim the second operation is given by an explicit formula
\nocite{*}
\bibliographystyle{alpha}
\bibliography{mfMSN}
\end{document}